\documentclass[11pt,reqno,twoside]{amsart}

\usepackage[english]{babel} 
\usepackage{esint}
\usepackage[T1]{fontenc} 
\usepackage{amsmath}
\usepackage{amssymb} 
\usepackage{float}
\usepackage{amsfonts}
\usepackage{mathtools}
\usepackage{xcolor}
\usepackage[utf8]{inputenc}
\usepackage[mathcal]{eucal} 
\usepackage{enumerate}
\usepackage{amsthm} 
\usepackage{hyperref}
\usepackage[totalwidth=14cm,totalheight=20cm, hmarginratio=1:1]{geometry}

\newcommand{\R}{\mathbb{R}}

\newcommand{\SO}{\mathrm{SO}}
\newcommand{\h}{\mathbb{H}}
\newcommand{\PSL}{\mathbb{P}\mathrm{SL}}
\newcommand{\Sp}{\mathrm{Sp}}
\newcommand{\T}{\mathcal{T}}
\newcommand{\Pp}{\mathbb{P}}
\newcommand{\Ind}{\mathrm{Ind}}
\newcommand{\SL}{\mathrm{SL}}
\newcommand{\Mix}{\mathrm{Mix}}
\newcommand{\Flat}{\mathrm{Flat}}
\newcommand{\Area}{\mathrm{Area}}
\newcommand{\sys}{\mathrm{sys}}
\newcommand{\diag}{\mathrm{diag}}

\renewcommand{\Re}{\operatorname{Re}}

\newtheorem{defi}{Definition}[section] 
\newtheorem{thm}[defi]{Theorem}
\newtheorem{cor}[defi]{Corollary} 
\newtheorem{lemma}[defi]{Lemma}

\newtheorem{prop}[defi]{Proposition}

\newtheorem{bigthm}{Theorem}

\theoremstyle{definition}
\newtheorem{rmk}{Remark}[section]

\DeclareMathAlphabet{\mathpzc}{OT1}{pzc}{m}{it}

\numberwithin{equation}{section}

\begin{document}




\title[Compactification $\SO_{0}(2,3)$-Hitchin component]{Length spectrum compactification of the $\SO_{0}(2,3)$-Hitchin component}
\author{Charles Ouyang and Andrea Tamburelli}

\begin{abstract}
We find a compactification of the $\SO_{0}(2,3)$-Hitchin component by studying the degeneration of the induced metric on the unique equivariant maximal surface in the 4-dimensional pseudo-hyperbolic space $\h^{2,2}$. In the process, we establish the closure in the space of projectivized geodesic currents of the space of flat metrics induced by holomorphic quartic differentials on a Riemann surface. As an application, we describe the behavior of the entropy of Hitchin representations along rays of quartic differentials.
\end{abstract}

\maketitle
\setcounter{tocdepth}{1}
\tableofcontents

\section*{Introduction and Main Results}

This article is the third in a series of papers (\cite{Charles_dPSL}, \cite{OT}) on the study of length spectrum compactifications of special connected components of representation varieties, called Hitchin components. These were first discovered by Hitchin using Higgs bundle techniques (\cite{Hitchin_selfdual}, \cite{Hitchin_Teichmuller}) and have been studied since then by many authors (\cite{Wienhard_ICM}) because of their beautiful geometric and dynamical properties (\cite{Guichard_Hitchin}, \cite{GW_PSL4}).
In particular, Labourie (\cite{Labourie_cyclic}) showed that for every Hitchin representation $\rho: \pi_{1}(S) \rightarrow G$ of the fundamental group of a closed surface $S$ into a semi-simple Lie group $G$ of rank $2$, there is a unique $\rho$-equivariant minimal embedding of the universal cover of $S$ into the symmetric space $G/K$. In the case of $G=\Pp\Sp(4,\R)$, this implied a parametrization of the $\Pp\Sp(4,\R)$-Hitchin component by the bundle $\mathcal{Q}^{4}(S)$ of holomorphic quartic differentials over the Teichm\"uller space $\T(S)$ of $S$. Later, exploiting the exceptional isomorphism $\Pp\Sp(4,\R)\cong \SO_{0}(2,3)$, a new interpretation of Labourie's result was found in terms of pseudo-Riemannian geometry (\cite{BTT}): the unique $\rho$-equivariant minimal immersion into the symmetric space arises as the Gauss map of the unique $\rho$-equivariant maximal surface (i.e with vanishing mean curvature) in the pseudo-hyperbolic space $\h^{2,2}$. Moreover, the data in $\mathcal{Q}^{4}(S)$ parametrizing $\SO_{0}(2,3)$-Hitchin representations correspond to the embedding data of the maximal surfaces: the point in Teichm\"uller space is the conformal class of the induced metric and the holomorphic quartic differential determines the second fundamental form. \\

In this paper, we start from this pseudo-Riemannian interpretation in order to describe a compactification of the $\SO_{0}(2,3)$-Hitchin component, generalizing Bonahon's approach to Thurston's compactification of Teichm\"uller space by way of geodesic currents. We consider the space $\Ind(S)$ of induced metrics on maximal surfaces in $\h^{2,2}$ that are equivariant under an $\SO_{0}(2,3)$-Hitchin representation and we prove the following:

\vspace{0.2cm}
\begin{bigthm}\label{thmA} The space of induced metrics $\Ind(S)$ is parametrized by $\mathcal{Q}^{4}(S)/S^{1}$ where $S^{1}$ acts by multiplication on the quartic differentials. Moreover, all metrics in $\Ind(S)$ are negatively curved.
\end{bigthm}
\vspace{0.2cm}

A similar phenomenon has been described for the induced metrics on minimal surfaces in $\h^{2}\times \h^{2}$ which are equivariant under pairs of discrete and faithful representations into $\PSL(2,\R)$ (\cite{Charles_dPSL}) and for the induced metrics on affine spheres in $\R^{3}$ in the context of $\SL(3,\R)$-Hitchin representations (\cite{OT}). However, this is the first case in which the differential equation determining the induced metric starting from the data $(\sigma, q) \in \mathcal{Q}^{4}(S)$ is a fully coupled system of nonlinear PDEs, which makes the analysis more challenging. \\

By a result of Otal (\cite{Otal}), we can then uniquely associate to any $g\in \Ind(S)$ a geodesic current $\mu_{g}$, which is a $\pi_{1}(S)$-invariant Radon measure on the space of unoriented biinfinite geodesics of $S$ encoding the length spectrum of $g$, that is, the collection of lengths of closed geodesics for $g$. Because the space of projectivized geodesic currents $\Pp\mathcal{C}(S)$ is compact, in order to find a compactification of $\Ind(S)$, and thus of the $\SO_{0}(2,3)$-Hitchin component, it is sufficient to describe the behavior of the length spectrum of induced metrics corresponding to diverging sequences of Hitchin representations up to scalar multiplication. We can show the following:

\vspace{0.1cm}
\begin{bigthm}\label{thmB} Let $(\sigma_{n}, q_{n})\in \mathcal{Q}^{4}(S)$ be a diverging sequence and let $g_{n}$ denote the corresponding sequence of induced metrics with associated geodesic currents $\mu_{n}$. Then there is a sequence of constants $c_{n}>0$ such that $c_{n}\mu_{n}$ converges, up to subsequences, to a mixed structure $\mu_{\infty} \in \Pp\Mix(S) \subset \Pp\mathcal{C}(S)$. Moreover, any mixed structure can be attained as the limit of a suitable diverging sequence.
\end{bigthm}
\vspace{0.1cm}

Here a mixed structure is a geodesic current that comes from a flat metric with cone angles $2\pi+\frac{k\pi}{2}$ on a subsurface and forms a measured lamination on the complement.\\

The main idea behind the proof of Theorem \ref{thmB} lies in a comparison between the induced metric on the maximal surface associated to the data $(\sigma, q)\in \mathcal{Q}^{4}(S)$ and the flat metric  $|q|^{\frac{1}{2}}$ with cone singularities. We show that for any diverging sequence in $\mathcal{Q}^{4}(S)$, their limiting geodesic currents are mixed structures that enjoy the same decomposition into subsurfaces and coincide in their non-laminar part. Thus Theorem \ref{thmA} follows once we establish that the closure in $\Pp\mathcal{C}(S)$ of the space of flat metrics with conic singularities $\Flat_{4}(S)$ coming from holomorphic quartic differentials is exactly $\Pp\Mix(S)$. This extends previous results for singular flat metrics coming from quadratic (\cite{DLR_flat}) and cubic differentials (\cite{OT}). \\

As an application of our estimates, we also describe the behavior of the volume entropy of the induced metric on the equivariant maximal surfaces along rays of quartic differentials:

\begin{bigthm}\label{thmC} Along a ray of quartic differentials on a fixed Riemann surface, the volume entropy of the induced metrics decreases monotonically to $0$.
\end{bigthm}

Again, similar results have been proved for the induced metrics on affine spheres (\cite{Nie_rays}) and equivariant maximal surfaces in anti-de Sitter space (\cite{Tambu_rays}), but this is the first time this analysis involves a system of PDEs. Moreover, it follows from Theorem \ref{thmC} that the volume entropy is maximal at the Fuchsian locus, thus recovering one of the results in \cite{BTT} with completely different techniques.

\subsection*{Acknowledgement} We would like to kindly thank Qiongling Li for communicating to us the proof of Proposition \ref{prop:neg_curv} and for her invitation to Nankai where some of this work was done. The initial stages of the project were carried out when both authors were visiting MSRI during the Fall 2019 program "Holomorphic Differentials in Mathematics and Physics": we thank the organizers and the institute for their hospitality. The first author acknowledges support from the National Science Foundation through grants DMS-1564374 and DMS-1745670. The second author acknowledges support from the National Science Foundation through grant DMS-2005501.

\section{Background}\label{sec:background}

\subsection{Higgs bundles}
In the study of Hitchin representations into rank 2 semi-simple Lie groups, $\Pp\Sp(4,\R)$ is the first instance where the study of a particular class of geometric objects follows chronologically from the analysis on the equations governing the Higgs bundle, namely the Hitchin equations. In the previous cases of $\PSL(2,\R)\times \PSL(2,\R)$ and $\SL(3, \R)$, the study of the associated geometric objects precedes  the study of their corresponding Higgs bundles. In the case of $\PSL(2,\R)\times \PSL(2,\R)$, the associated geometric objects are minimal Lagrangians in $\h^{2}\times \h^{2}$, which was first studied by Schoen (\cite{Schoenharmonic}) as part of a way to define a midpoint map on Teichm\"uller space using harmonic maps. The Hitchin equations in this setting are precisely the Bochner formulas for harmonic maps between hyperbolic surfaces (see \cite{Hitchin_selfdual}, \cite{Wolf_harmonic}). For $\SL(3, \R)$, the geometric objects are affine spheres in $\R^{3}$, the study of which dates back to \textcommabelow{T}i\textcommabelow{t}eica, followed by Blaschke, Calabi, Cheng-Yau, and others (see \cite{loftin2010survey}). In the Higgs bundle interpretation, the Hitchin equations in this setting are the structural equations for hyperbolic affine spheres in $\R^{3}$. As Higgs bundles play the motivating role of our study of equivariant maximal surfaces in $\h^{2,2}$, we provide a cursory introduction to the theory of Higgs bundles. Many excellent references in the literature exist (\cite{guichard2018introduction},\cite{QL_survey}, \cite{wentworth2016higgs}, \cite{Wienhard_ICM}), and we make no attempt to count this section among them.\\

Let $X$ be a Riemann surface structure on $S$. If $K_{X}$ denotes its canonical bundle, then a \textit{Higgs bundle} is a pair $(\mathcal{E}, \Phi)$, where $\mathcal{E}$ is a holomorphic vector bundle on $X$ and $\Phi$, called the \textit{Higgs field}, is a holomorphic section of the bundle $\mathrm{End}(\mathcal{E}) \otimes K_{X}$.
We will primarily be concerned with Higgs bundles which are both (poly)-stable and cyclic. We say a Higgs bundle $(\mathcal{E}, \Phi)$ is \textit{stable} if for every non-zero $\Phi$-invariant subbundle $\mathcal{F}$, the slope condition
$$\frac{\mathrm{deg}(\mathcal{F})}{\mathrm{rank}(\mathcal{F})} < \frac{\mathrm{deg}(\mathcal{E})}{\mathrm{rank}(\mathcal{E})} $$
is satisfied. A Higgs bundle is said to be \textit{polystable} if it is a direct sum of stable Higgs bundles.\\

To the pair $(\mathcal{E}, \Phi)$, the Hitchin equations are the coupled pair of equations
\begin{align}
\begin{split}
&\overline{\partial}_{\mathcal{E}} \Phi = 0 \\
&F_{\nabla_{\overline{\partial}_{\mathcal{E}}, H}} + [\Phi, \Phi^{*_{H}}]= 0,
\end{split}
\end{align}
where $\overline{\partial}_{\mathcal{E}}$ is the complex structure on the bundle $\mathcal{E}$, and $H$ is a Hermitian metric on $\mathcal{E}$. Here $\nabla_{\overline{\partial}_{\mathcal{E}}, H}$ is the Chern connection, which is the unique connection compatible with $H$ 
and whose (0,1) part is the complex structure on $\mathcal{E}$. 
The (poly)stability condition guarantees that there is a solution to the Hitchin equations (\cite{Simpson_Higgs}, \cite{Hitchin_selfdual}). \\

A Higgs bundle $(\mathcal{E}, \Phi)$ is said to be $\textit{cyclic}$ if the bundle $\mathcal{E}$ splits holomorphically into a sum of line bundles, that is,
$$ \mathcal{E} = \mathcal{L}_{1} \oplus \mathcal{L}_{2} \oplus ... \oplus \mathcal{L}_{n} $$
and the Higgs field is of the form
$$ \Phi = \begin{pmatrix}
0 &  &  & \gamma_{n}\\
\gamma_{1} & 0 &  & \\
 & \ddots & \ddots & \\
 & & \gamma_{n-1} & 0
\end{pmatrix}.$$
In particular, a Higgs field may be regarded as a bundle map $\Phi: \mathcal{E} \to \mathcal{E} \otimes K$, so the cyclicity condition says the Higgs field is a collection of graded bundle maps $\gamma_{i}: \mathcal{L}_{i} \to \mathcal{L}_{i+1} \otimes K$, for $1\leq i \leq n-1$ and $\gamma_{n}: \mathcal{L}_{n} \to \mathcal{L}_{1} \otimes K$. For cyclic Higgs bundles, the Hitchin equations are more treatable: one can show (\cite{Baraglia_cyclic}, \cite{Collier_thesis}) that the Hermitian metric $H$ satisfying the Hitchin equations must be diagonal $H=\diag(h_{1},h_{2},\dots, h_{m}, h_{m}^{-1}, \cdots, h_{1}^{-1})$ (we assume here $n=2m$, which is the case of interest in this paper) in the holomorphic splitting described above and the Hitchin equations reduce to the following system of PDEs
\begin{equation}\label{eq:cyclic}
\begin{cases}
    \Delta \log(h_{1})+|\gamma_{1}|^{2}h_{1}^{-1}h_{2}-|\gamma_{n}|^{2}h_{1}^{2}=0 \\
    \Delta \log(h_{k})+|\gamma_{k}|^{2}h_{k}^{-1}h_{k+1}-|\gamma_{k-1}|h_{k-1}^{-1}h_{k}=0 \ \ \  \ \ \text{for $k=2,\dots, m-1$} \\
    \Delta \log(h_{m})+|\gamma_{m}|^{2}h_{m}^{-2}-|\gamma_{m-1}|^{2}h_{m-1}^{-1}h_{m}=0
\end{cases}
\end{equation}
where $\gamma_{i}$ is holomorphic. Here the equations are written in local coordinate charts $U$ of $S$ and we have denoted by $\Delta$ the operator $\bar{\partial}\partial$, which is one quarter of the usual Laplacian.

\subsection{The non-abelian Hodge correspondence}
The non-abelian Hodge correspondence furnishes a dictionary between representations and Higgs bundles by way of harmonic maps. In his seminal paper \cite{Hitchin_selfdual}, Hitchin introduced the notion of a Higgs bundle and the Hitchin equations in order to study surface group representations into $\mathbb{P} \SL (2, \mathbb{C})$, though the correspondence can be shown to hold in greater generality. Here, we describe the non-abelian Hodge correspondence for $G= \textnormal{GL}(n, \mathbb{C})$ (for particular subgroups of $\textnormal{GL}(n, \mathbb{C})$, additional properties must be imposed on the Higgs bundle: for instance if $G = \SL(n, \mathbb{C})$, the trace of the Higgs field $\phi$ must vanish; in the present setting where $G=\textnormal{Sp}(4, \R)$, the associated flat connection $\nabla$ will preserve a real subbundle $\mathcal{E}_{\R}$ and a symplectic form $\omega$, see Section 1.4). \\

The representation variety $\textnormal{Hom}(\pi_{1}(S),\textnormal{GL}(n, \mathbb{C}))/ \textnormal{GL}(n, \mathbb{C}) $ consists of representations of the fundamental group to the group $G =\textnormal{GL}(n, \mathbb{C})$ up to conjugation. From such a representation $\rho$, one may form an associated smooth rank $n$ complex vector bundle as follows: first form the trivial bundle $\widetilde{S} \times \mathbb{C}^{n}$, then take the quotient $E_{\rho} := \widetilde{S} \times_{\rho} \mathbb{C}^{n}$ by the identification $(\widetilde{x}, v) \sim (\gamma \cdot \widetilde{x}, \rho(\gamma) \cdot v)$, for $\gamma \in \pi_{1}(S)$. The flat structure on the trivial bundle descends to a flat structure on $E_{\rho}$, and hence equips $E_{\rho}$ with a flat connection $\nabla$. From a flat connection on a rank $n$ complex vector bundle $E$ over $S$, one may recover the (conjugacy class of the) representation by taking its holonomy representation. The flatness condition guarantees the holonomy depends solely on the homotopy class of the loop, and not just the loop itself.\\

Given a Hermitian metric on $(E, \nabla)$, one may uniquely write the connection $\nabla$ as a sum of a unitary connection $\nabla_{H}$ and a $H$-Hermitian 1-form $\psi_{H}$ with coefficients in End$(E)$. If a Riemann surface structure $X$ is chosen on $S$, then one may further decompose $\nabla_{H}$ and $\psi_{H}$ by type. In short, one has
\begin{align*}
    \nabla &= \nabla_{H} + \psi_{H}\\
    &= \nabla_{H}^{1,0} + \nabla_{H}^{0,1} + \psi_{H}^{1,0} + \psi_{H}^{0,1}.
\end{align*}
Work of Donaldson and Corlette shows how one can construct a stable Higgs bundle from this data.
\begin{thm}[\cite{Donaldson_selfdual}, \cite{Corlette}]
If $\nabla$ is a flat irreducible $\mathrm{GL}(n, \mathbb{C})$ connection on a vector bundle $E$ over $X$, then there exists a unique (up to a scalar multiple) Hermitian metric $H$
so that $(E, \nabla_{H}^{0,1}, \psi_{H}^{1,0})$ is a stable $\mathrm{GL}(n, \mathbb{C})$-Higgs bundle over $X$.
\end{thm}

The theorem is proven using harmonic maps, and the role of the Riemann surface structure chosen along with the associated Hermitian metric and the harmonic map will play an important role in the sequel and so we make a brief digression to discuss the basics of harmonic maps. Recall a harmonic map $f: (M,g) \to (N,h)$ between two closed Riemannian manifolds is a Lipchitz map which is a critical point of the energy functional
$$ E(f) = \frac{1}{2} \int_{M} |df|^{2} dA_{g}. $$
When the domain is a surface, the energy of a map is invariant when replacing the domain metric with a different metric in the same conformal class. Hence, harmonic maps are well-defined when the domain is a Riemann surface.\\

A Hermitian metric on $E_{\rho}$ is equivalent to a $\rho$-equivariant map from $\widetilde{X}$ into the symmmetric space $\textnormal{GL}(n, \mathbb{C}) / \textnormal{U}(n)$. Donaldson and Corlette's theorem may thus be rephrased as follows.

\begin{thm}[\cite{Donaldson_selfdual}, \cite{Corlette}]
Let $\nabla$ be a flat irreducible $\mathrm{GL}(n, \mathbb{C})$-connection on a vector bundle $E$ over $X$ with holonomy representation $\rho: \pi_{1}(S) \to\textnormal{GL}(n, \mathbb{C})$, there there exists a unique $\rho$-equivariant harmonic map $f_{\rho}: \widetilde{X} \to \textnormal{GL}(n, \mathbb{C}) / \textnormal{U}(n)$.

\end{thm}

In light of the preceding theorem, the Hermitian metric $H$ is sometimes referred to as the harmonic metric. Finally, to complete the correspondence, the following result of Hitchin and Simpson shows how to reconstruct a flat connection given a Higgs bundle.

\begin{thm}[\cite{Hitchin_selfdual}, \cite{Simpson_Higgs}]
Let $(\mathcal{E}, \varphi)$ be a polystable $\textnormal{GL}(n, \mathbb{C})$-Higgs bundle, then there exists a unique Hermitian metric $H$ on $\mathcal{E}$ such that
$$ \nabla = \nabla_{\overline{\partial}_{\mathcal{E}}, H} + \varphi + \varphi^{*_{H}}$$
is flat, where $\nabla_{\overline{\partial}_{\mathcal{E}}, H}$ is the Chern connection determined by $H$ and $\overline{\partial}_{\mathcal{E}}$, and $\varphi^{*_{H}}$ is the Hermitian adjoint of $\phi$.
\end{thm}
The unique Hermitian metric, which solves the Hitchin equations (1.1) is the harmonic metric discussed previously, once again showing the importance of the equivariant harmonic maps.\\

We conclude our discussion here by highlighting an important class of representations. It is well known that there is a unique irreducible representation $\iota: \PSL(2, \R) \hookrightarrow \PSL(n, \R),$ where $n \geq 2$. Thus to any Fuchsian representation $\rho$, its image under postcomposition by $\iota$ is an element of the $\PSL(n, \R)$ representation variety. In fact, through postcomposition by $\iota$, Teichm{\"u}ller space injects into a distinguished component of the $\PSL(n, \R)$ character variety (and hence the moduli space of Higgs bundle via the non-abelian Hodge correspondence), called the \textit{Hitchin component}. Elements of this particular distinguished component are called \textit{Hitchin representations}, and so may be thought of as smooth deformations of Fuchsian representations. 

In the present setting of $n=4$, the image of $\PSL(2, \R)$ by $\iota$ lies entirely in $\Pp\Sp(4,\R)$. Under the non-abelian Hodge correspondence, the associated Higgs bundles will have a rather nice expression (see Section 1.4).

\subsection{Labourie's conjecture}

In defining a Higgs bundle, a Riemann surface structure must first be placed on the underlying smooth surface $S$. The choice of the Riemann surface $X$ may be arbitrary as the correspondence between the representation variety and the moduli space of Higgs bundles is constructed through analysis of the unique $\rho$-equivariant harmonic map $f$ from $\widetilde{X}$ to the symmetric space $G/K$. A choice of a different Riemann surface structure $X'$ yields a different unique $\rho$-equivariant harmonic map $f'$ from $\widetilde{X'}$ to $G/K$, from which the Higgs bundle is constructed. \\ 

Harmonicity of the map $f$ produces a holomorphic object, known as the \textit{Hopf differential}, which is the (2,0)-part of the pullback metric $f^{*}h$, and so is a holomorphic quadratic differential. When the map is both harmonic and conformal, the corresponding Hopf differential vanishes.\\

To this end, Labourie has conjectured there is a natural choice of a Riemann surface structure for each representation. To each Hitchin representation $\rho$ into a split real Lie group $G$, there is a unique Riemann surface structure $X$, so that the unique $\rho$-equivariant harmonic map $f_{\rho} : \widetilde{X} \to G/K$ is conformal. The conjecture has been completely resolved in rank 2. For example, for the group $G=\PSL(2,\R)\times \PSL(2,\R)$, the conjecture is equivalent to the uniqueness of minimal lagrangians in the bidisc $\mathbb{H}^{2} \times \mathbb{H}^{2}$, first studied by Schoen \cite{Schoenharmonic}.\\

Conformality of the unique $\rho$-equivariant harmonic map ensures the image is a $\rho$-equivariant minimal surface inside the symmetric space $G/K$. Hence Labourie's conjecture produces minimal surfaces inside the associated symmetric space from the data of Hitchin representations. These geometric objects bridge the gap between surface group representations, which are topological and algebraic objects, and Higgs bundles, which are holomorphic objects.\\

The mapping class group acts on the space of conjugacy classes of Hitchin representations by remarking. In the absence of a preferred complex structure to each such Hitchin representation, the parameterization of Hitchin components using an arbitrary Riemann surface surface structure does not admit a natural mapping class group action. Labourie's conjecture, when true, restores this symmetry, which is broken in the non-abelian Hodge correspondence. \\

Finally, to each Hitchin representation $\rho$ in a rank-two Lie group $G$, if the corresponding stable Higgs bundle is constructed using the unique conformal structure $X_{\rho}$, then the Higgs bundle is cyclic. Notice that \textit{a priori}, a Higgs bundle constructed from a Hitchin representation may also be cyclic for several different Riemann surfaces, but Labourie \cite{Labourie_cyclic} shows that in rank two, this cannot happen; the Higgs bundle is cyclic only for a unique Riemann surface.

\subsection{Maximal surfaces in pseudo-hyperbolic space} We now specialize to Higgs bundles in the $\Sp(4,\R)$-Hitchin component and explain their relation with maximal surfaces in the pseudo-hyperbolic space $\h^{2,2}$. \\

Let $X$ be a Riemann surface structure on $S$ and let $K=K_{X}$ be its canonical bundle. An $\Sp(4,\R)$-Higgs bundle in the Hitchin component is a pair $(\mathcal{E}, \varphi)$ where 
\[
\mathcal{E}=K^{\frac{3}{2}}\oplus K^{-\frac{1}{2}}\oplus K^{-\frac{3}{2}}\oplus K^{\frac{1}{2}} \ \ \ \text{and} \ \ \ 
    \varphi=\begin{pmatrix}
            0 & 0 & q_{4} & q_{2} \\
            0 & 0 & q_{2} & 1 \\
            0 & 1 & 0 & 0 \\
            1 & q_{2} & 0 & 0 
            \end{pmatrix}
\]
with $q_{2}$ and $q_{4}$ being holomorphic quadratic and quartic differentials respectively. By Labourie's work (\cite{Labourie_cyclic}), we can assume that the complex structure $X$ has been chosen so that the associated equivariant harmonic map is conformal and that $q_{2}$ vanishes identically. In this way, we obtain a cyclic $\Sp(4,\R)$-Higgs bundle which only depends on the Riemann surface $X$ and the holomorphic quartic differential $q=q_{4}$. Therefore, the $\Sp(4,\R)$-Hitchin component is parametrized, via the non-abelian Hodge correspondence, by the bundle $\mathcal{Q}^{4}(S)$ of holomorphic quartic differentials over the Teichm\"uller space of $S$. We already mentioned that, in the cyclic setting, the Hermitian metric $H$ that solves the Hitchin equations is diagonal. If we write $H=\diag(h_{1}, h_{2}^{-1}, h_{1}^{-1}, h_{2})$, then $h_{1}$ and $h_{2}$ are local solutions of the system of PDEs
\begin{equation}\label{eq:cyclic_sp4}
\begin{cases}
    \Delta \log(h_{1})+h_{1}^{-1}h_{2}-|q|^{2}h_{1}^{2}=0 \\
    \Delta \log(h_{2})+h_{2}^{-2}-h_{1}^{-1}h_{2}=0 \ .
\end{cases}
\end{equation}
It is convenient to rewrite these equations with respect to a background hyperbolic metric $2\sigma$ in the conformal class. We regard $\sigma$ as a section of $K\otimes \overline{K}$ so that in local coordinates $\sigma=\sigma(z)dz\otimes d\bar{z}$ and the corresponding Riemannian metric on $S$ is locally $2\sigma=2\sigma(z)(dx^{2}+dy^{2})$. The functions $\psi_{1}$ and $\psi_{2}$ such that
\[
    h_{1}^{-1}=e^{\psi_{1}}\sigma^{\frac{3}{2}} \ \ \ \text{and} \ \ \ 
    h_{2}^{-1}=e^{\psi_{2}}\sigma^{\frac{1}{2}}
\]
are globally defined on $S$ and solve the equations
\begin{equation}\label{eq:cyclic_sp4_global}
    \begin{cases}
    \Delta_{\sigma} \psi_{1} =e^{\psi_{1} - \psi_{2}}-e^{-2\psi_{1}} \|q\|^{2}_{\sigma} +\frac{3}{4}\kappa(\sigma) \\ 
    \Delta_{\sigma} \psi_{2} =e^{2\psi_{2}}- e^{\psi_{1} - \psi_{2}} + \frac{1}{4}\kappa(\sigma)
    \end{cases} \ . 
\end{equation}

\noindent A solution $H$ to Hitchin equations gives a flat connection $\nabla=\nabla_{\bar{\partial}_{\mathcal{E}},H}+\varphi+\varphi^{*H}$ with holonomy in $\Sp(4,\R)$. This means, in particular, that $\nabla$ preserves a real subbundle $\mathcal{E}_{\R}$ of $\mathcal{E}$ and a symplectic form on $\mathcal{E}_{\R}$, that we are now going to describe briefly. The bundle $\mathcal{E}$ is equipped with the orthogonal structure 
\[
    Q=\begin{pmatrix}
        0 & \mathrm{Id} \\
        \mathrm{Id} & 0 
        \end{pmatrix} \ ,
\]
which induces a real involution $\tau:\mathcal{E}\rightarrow \overline{\mathcal{E}}$ given by $\tau(v)=H^{-1}Q\bar{v}$. The subbundle $\mathcal{E}_{\R}=\mathrm{Fix}(\tau)$ is parallel for the flat connection $\nabla$ and is generated by the local sections
\[
    u_{1}(z)=\frac{1}{\sqrt{2}}\begin{pmatrix}
            0 \\ h_{2}^{\frac{1}{2}} \\ 0 \\ h_{2}^{-\frac{1}{2}} 
            \end{pmatrix}
    \ \ \ 
    u_{2}(z)=\frac{1}{\sqrt{2}}\begin{pmatrix}
            h_{1}^{-\frac{1}{2}} \\ 0 \\ h_{1}^{\frac{1}{2}} \\ 0 
            \end{pmatrix}
\]
\[
    u_{3}(z)=\frac{1}{\sqrt{2}}\begin{pmatrix}
            0 \\ ih_{2}^{\frac{1}{2}} \\ 0 \\ -ih_{2}^{-\frac{1}{2}} 
            \end{pmatrix}
    \ \ \ 
    u_{4}(z)=\frac{1}{\sqrt{2}}\begin{pmatrix}
            ih_{1}^{-\frac{1}{2}} \\ 0 \\ -ih_{1}^{\frac{1}{2}} \\ 0 
            \end{pmatrix} \ .
\]
Moreover, the connection $\nabla$ preserves the symplectic form $\omega_{\R}$ on $\mathcal{E}_{\R}$, which in the local frame $\{u_{1}, u_{2}, u_{3}, u_{4}\}$ is given by
\[
    \omega_{\R}=\begin{pmatrix}
            0 & -\mathrm{Id} \\
            \mathrm{Id} & 0 
            \end{pmatrix} \ ,
\] 
and consequently $\nabla$ preserves the volume form $-\frac{1}{2}\omega_{\R}\wedge \omega_{\R}=u_{1}\wedge u_{2}\wedge u_{3}\wedge u_{4}$ on $\mathcal{E}_{\R}$.
This data induces a $\nabla$-parallel bilinear form of signature $(3,3)$ on the bundle $\Lambda^{2}\mathcal{E}_{\R}$ of skew-symmetric 2-tensors on $\mathcal{E}_{\R}$ by setting
\[
    -2\langle \phi, \psi \rangle u_{1}\wedge u_{2}\wedge u_{3}\wedge u_{4}= \phi \wedge \psi \ 
\]
for every $\phi, \psi \in \Lambda^{2}\mathcal{E}_{\R}$. Let $\omega^{*}$ be the skew-symmetric 2-tensor defined by the property that 
\[
    -2\langle \omega^{*}, v \wedge w \rangle = \omega_{\R}(v,w)
\]
for every $v,w \in \mathcal{E}_{\R}$. Because $\omega_{\R}$ is parallel for $\nabla$, the  
5-dimensional subbundle $V \subset \Lambda^{2}\mathcal{E}_{\R}$ obtained by taking the orthogonal to $\omega^{*}(z)$ at every $z\in S$ is preserved by $\nabla$ and is naturally endowed with a parallel bilinear form of signature $(2,3)$. Consider now the section $f(z)=u_{1}(z)\wedge u_{3}(z)-u_{2}(z)\wedge u_{4}(z)$, which takes values into the space of unit time-like vectors of $V$. If we denote by $\rho:\pi_{1}(S) \rightarrow \Sp(4,\R)$ the holonomy of the flat connection $\nabla$, the section $f$ induces a $\rho$-equivariant map $\tilde{f}:\tilde{S} \rightarrow \h^{2,2}$ by fixing a base point $z_{0} \in \tilde{S}$ and setting $\tilde{f}(z)$ to be the $\nabla$-parallel transport of a lift of $f$ to the fiber over $z_{0}$. Here we identified $\h^{2,2}$ with the set of unit time-like vectors of $V_{z_{0}}$. It can be shown (\cite{TW}, \cite{Baraglia_thesis}) that $\tilde{f}$ is conformal and harmonic and thus parametrizes a maximal surface in the pseudo-hyperbolic space $\h^{2,2}$. Moreover, from the computations in \cite[Proposition 4.5]{TW}, we know that the induced metric on the maximal surface $\tilde{f}(\tilde{S})/\rho(\pi_{1}(S))$ is $g=4h_{1}^{-1}h_{2}=4e^{\psi_{1}-\psi_{2}}\sigma$, where $\psi_{1}$ and $\psi_{2}$ are the solutions to Equation (\ref{eq:cyclic_sp4_global}).

\subsection{Geodesic currents}
Historically, the notion of a geodesic current was first defined by Bonahon (\cite{Bonahon_currents}, \cite{bonahonbouts}) as a way to study the Thurston compactification of Teichm{\"u}ller space. In higher Teichm{\"u}ller theory, geodesic currents have been used extensively to study limits of representations (see \cite{BIPP}, \cite{burger2020real}, \cite{martone2019sequences}, \cite{Charles_dPSL}, \cite{OT}). Here we provide a cursory review of geodesic currents.\\

Let $S$ be a closed surface of negative Euler characteristic, and let $S$ be endowed with a fixed auxiliary hyperbolic metric $\sigma$. By the Cartan-Hadamard theorem, the universal cover $(\widetilde{S}, \widetilde{\sigma})$ is isometrically diffeomorphic to $\mathbb{H}^{2}$. Consider then the space $\mathpzc{G}(\widetilde{S}, \widetilde{\sigma})$ of unoriented biinfinite geodesics of $(\widetilde{S}, \widetilde{\sigma})$. It is not difficult to see that
$$\mathpzc{G}(\widetilde{S}, \widetilde{\sigma}) \cong( (S^{1} \times S^{1}) - \Delta) \backslash \mathbb{Z}_{2}.$$
A \textit{geodesic current} is a $\pi_{1}(S)$-equivariant Radon measure on $\mathpzc{G}(\widetilde{S}, \widetilde{\sigma})$. The space $\mathcal{C}(S)$ of geodesic currents is endowed with the weak-$*$ topology, that is, $\mu_{n} \to \mu$ if and only if
$$\int_{\mathcal{G}(\widetilde{S})} f d\mu_{n} \to \int_{\mathcal{G}(\widetilde{S})} f d \mu$$
for any continuous function $f$ on $\mathpzc{G}(\widetilde{S},\tilde{\sigma})$ of compact support. While it may appear that the space of geodesic currents strongly depends upon the initial choice of an auxiliary metric, we remark that if a different hyperbolic were chosen, the lifts of the two metrics to their universal covers would be $\pi_{1}(S)$-equivariantly quasi-isometric to one another, and this quasi-isometry would extend to their ideal boundaries. Hence, there would be a $\pi_{1}(S)$-equivariant homeomorphism between the two spaces of bi-infinite unoriented geodesics. \\

A natural first example of a geodesic current is obtained by taking a single closed geodesic $\gamma$ on $(S, \sigma)$ and lifting it to $(\widetilde{S}, \widetilde{\sigma})$. The Dirac measure supported on the lifts of $\gamma$ is a geodesic current, denoted $\delta_{\gamma}$. This construction may be repeated to show that weighted multicurves and measured laminations also give rise to geodesic currents. Moreover, weighted closed curves are dense in the space of geodesic currents. \\

The space of geodesic currents comes equipped with a natural continuous bilinear form $i: \mathcal{C}(S) \times \mathcal{C}(S) \to \mathbb{R}_{\geq 0}$, which restricts to Thurston's intersection form for measured laminations. In particular, when evaluated on pairs of geodesic currents $\delta_{\alpha}$ and $\delta_{\beta}$ associated to simple closed curves, the bilinear pairing $i$ gives the geometric intersection number of $\alpha$ and $\beta$. Moreover, measured laminations can be characterized as geodesic currents with vanishing self-intersection number (\cite{Bonahon_currents}).\\

Let $C(S)$ denote the set of isotopy classes of closed curves on S. Then by the preceding discussion, the set $C(S)$ may be regarded as a subset of $\mathcal{C}(S)$. The marked length spectrum of a geodesic current $\mu$ is the list of pairs 
$$ \{\gamma, i(\mu, \delta_{\gamma}) \}_{\gamma \in C(S)}. $$
Here, the marking comes from recording the data of the curve $\gamma$ alongside the intersection number $i(\mu, \delta_{\gamma})$. The unmarked length spectrum lists solely the intersection numbers, without any mention of the curve $\gamma$. In what follows, by length spectrum of a geodesic current $\mu$, it will be understood that it is the marked length spectrum of $\mu$. The distinction is a fairly important one, as marked length spectra distinguish geodesic currents, while the unmarked spectrum does not (\cite{Otal}). \\

An important facet of geodesic currents is that it contains a host of several geometric objects. Historically, the first non-trivial geodesic currents constructed were those coming from hyperbolic metrics, called Liouville currents (\cite{Bonahon_currents}), whose length spectrum is given by the hyperbolic length of the unique geodesic representative of $\gamma \in \pi_{1}(S)$.  Later it was shown (\cite{Otal}) that more generally, all negatively curved Riemannian metrics on $S$ can be realized as geodesic currents in the above sense. The self-intersection number in this case is always equal to $\frac{\pi}{2}\Area(S)$ (\cite{Otal}). More recently, this construction has been generalized to locally CAT(-1) metrics on $S$ (\cite{Hers_Paulin}) and non-positively curved Riemannian metrics with conic singularities (\cite{Frazier_thesis}, \cite{Rigidity_flat}).\\

The space $\mathbb{P}\mathcal{C}(S)$ of projectivized currents is defined by removing the zero current from $\mathcal{C}(S)$ and identifying currents that are positive multiple of each other. Bonahon (\cite{Bonahon_currents}) proved the space $\mathbb{P}\mathcal{C}(S)$ of projectivized currents is compact. 

\section{Length spectrum compactification of quartic differential metrics}

\subsection{Quartic differential metrics}
 Let $q$ be a holomorphic quartic differential over a Riemann surface $X$. In the local coordinates $z$, the differential $q$ may be written locally as $q(z) dz^{4}$, for some holomorphic function $q(z)$ on $X$. A singular flat metric $|q|^{\frac{1}{2}}$ is obtained from $q$, by first taking the absolute value $|q|$ and then taking the square root. The metric is smooth away from the zeros of $q$. Holomorphicity of the differential away from the zeros is precisely the condition that the curvature of $|q|^{\frac{1}{2}}$ is zero. At a zero of $q$ of order $k$, the metric $|q|^{\frac{1}{2}}$ has a conic singularity of angle $2\pi + \frac{k \pi}{2}$. However, there is a converse to the construction above. Given any flat metric $m$ with conic singularities with angle $2\pi + \frac{k \pi}{2}$ and holonomy in the subgroup of $\SO(2)$ generated by a rotation of angle $\frac{\pi}{2}$, there is a complex structure on $S$ and a holomorphic quartic differential $q$ so that $|q|^{\frac{1}{2}} = m$  (\cite{strata}).\\


Let $\mathcal{Q}^{4}(S)$ denote the bundle of holomorphic quartic differentials over Teichm{\"u}ller space. Then there is a natural circle action on $\mathcal{Q}^{4}(S)$, which fixes the Riemann surface structure, but multiplies a quartic differential by $e^{i \theta}$. If Flat$_{4}({S})$ denotes the space of quartic differential metrics on $S$, then Flat$_{4}({S}) \cong \mathcal{Q}^{4}(S)/S^{1}$, as $|q| = |e^{i\theta} q|$. Likewise, if Flat$_{4}^{1}({S})$ denotes the space of unit area quartic differential metrics, then Flat$_{4}^{1}({S}) \cong \mathcal{Q}^{4}(S)/\mathbb{C}^{*}$. Notice that the space Flat$_{2}^{1}({S})$ of unit area quadratic differential metrics is a proper subset of Flat$_{4}^{1}({S})$; from a quadratic differential $r$, the differential $r \otimes r$ is a quartic differential. 

\subsection{Geometric convergence of quartic differentials}

In this section we review the notion of geometric convergence for Riemann surfaces and holomorphic quartic differentials. This will closely follow work of McMullen \cite{McMullen_Poincareseries} in the setting of holomorphic quadratic differentials. Indeed, most of the results generalize immediately to higher order differentials. However, some care will be required to obtain the analogue of Proposition A.3.2 of \cite{McMullen_Poincareseries}. Extending these results specifically to the case of holomorphic cubic differentials was done in \cite{OT}.\\

Let $\kappa \in [-1, 0)$ and consider the metric
$$ g_{\kappa} = \left(\frac{4}{4+\kappa |z|^{2}} \right)^{2} |dz|^{2}, $$
which is a complete metric of constant curvature $\kappa$ on the domain $U_{\kappa} = \{z \in \mathbb{C} : |z| < R\}$, where $R= \frac{2}{\sqrt{- \kappa}}$. For $\kappa =0$, the metric $g_{\kappa}$ is a complete metric on the entire complex plane $\mathbb{C} = U_{0}$. In both instances, the pair $(U_{\kappa}, g_{\kappa})$ is a complete Riemannian manifold.\\

Let $\mathcal{X}$ denote the space of pairs $(U_{\kappa} , \Gamma)$, where $\Gamma$ is a discrete subgroup of the group of M{\"o}bius transformations acting freely on $U_{\kappa}$. From the data of $(U_{\kappa}, \Gamma)$, one may construct a framed Riemannian manifold as follows: the quotient $X = U_{\kappa} / \Gamma$ is a complete Riemannian manifold; the image of the unit vector in $U_{\kappa}$ at the origin pointing along the real axis, gives a framing $v$ on $X$. We may rescale the metric to ensure the injectivity radius at $v$ is at least 1. As $X$ comes endowed with a natural complex structure (independent of scaling), the aforementioned construction gives a framed Riemann surface $(X,v)$ with injectivity radius at least 1 at $v$. Conversely, given a Riemann surface $X$ admitting a complete metric of constant curvature $\kappa$ along with a framing $v$ with injectivity at least one, this uniquely determines an element of $\mathcal{X}$.\\

The \textit{geometric topology} on $\mathcal{X}$ is defined as follows: a sequence of pairs $(U_{\kappa_{n}}, \Gamma_{n})$ converges to $(U_{\kappa}, \Gamma)$ if and only if $\kappa_{n} \to \kappa$ and $\Gamma_{n} \to \Gamma$ in the Hausdorff topology on closed subsets of $\mathbb{P}\SL(2, \mathbb{C})$. We record the following well-known results.

\begin{prop}
The space $\mathcal{X}$ endowed with the geometric topology is compact.
\end{prop}

If $\mathcal{X}_{g,n} \subset \mathcal{X}$ denotes the space of finite type Riemann surfaces of genus $g$ with $n$ punctures, one may obtain its compactification by taking its closure in $\mathcal{X}$. 

\begin{prop}
For $n > 0$,
$$\overline{\mathcal{X}_{g,n}} = \bigcup\, \{\mathcal{X}_{h,m}: 2h+m \leq 2g+n, 1 \leq m \},$$
while for $n=0$ and g>0,
$$\overline{\mathcal{X}_{g,0}} = \mathcal{X}_{g,0} \cup \overline{\mathcal{X}_{g-1, 2}},$$
and
$$\overline{\mathcal{X}_{0,0}} = \mathcal{X}_{0,0} \cup \mathcal{X}_{0,1}.$$
\end{prop}

The hyperbolic Riemann surfaces are the finite type Riemann surfaces of genus $g$ with $n$ punctures for which $2-2g-n <0$. The compactification result then says the hyperbolic Riemann surfaces are compactified by the plane, the punctured plane, and hyperbolic surfaces of lower complexity.\\

The bundle $\mathcal{Q}^{4}$ of holomorphic quartic differentials over $\mathcal{X}$ is the space of triples $(U_{\kappa},\Gamma, q)$, where $(U_{\kappa},\Gamma) \in \mathcal{X}$ and $q$ is a holomorphic quartic differential for the Riemann surface $U_{\kappa}/\Gamma$. This space is endowed with the geometric topology as follows: a sequence of holomorphic quartic differentials $q_{n}$ converges geometrically to $q$ if and only if the domain Riemann surfaces converge geometrically in $\mathcal{X}$ and the lifts of $q_{n}$ to the universal covers converge uniformly on compact sets to the lift of $q$. We denote by $\mathcal{Q}_{g}^{4}\subset \mathcal{Q}^{4}$ the subspace of holomorphic quartic differentials on a surface of genus $g$. The main result we need for the rest of the paper is the following
\begin{thm}\label{thm:geometric_limit} The space $\mathbb{P}\mathcal{Q}_{g}^{4}$ has compact closure in $\mathbb{P}\mathcal{Q}^{4}$. 
\end{thm}

We leave the details of the proof for the interested reader, as it follows the same arguments of \cite[Appendix]{McMullen_Poincareseries} and \cite[Appendix]{OT}. In both papers there is only one technical statement (\cite[Proposition A.3.2]{McMullen_Poincareseries} and \cite[Proposition 4.10]{OT}) about the linear independence of Poincar\'e series of quadratic and cubic differentials, where hyperbolic and convex real-projective geometry are used. Here we explain how to obtain the analogous result for quartic differentials.\\

Let us first introduce some notation. Let $(Y,w)\in \mathcal{X}$ be a framed Riemannian surface with constant curvature $\kappa$ and fundamental group $\mathbb{Z}$. Standard models for these surfaces are described in \cite[Section 4.3]{OT}. Choose a biholomorphism 
\[
h:Y \rightarrow A(r,R) =\{z \in \mathbb{C} \ | \ r<|z|<R\}
\]
with $0\leq R \leq +\infty$. We set $\phi(Y,w) =h^{*}\left(\frac{dz^{4}}{z^{4}}\right)$. Let $S$ be a finite system of pairwise distinct, non-trivial isotopy classes of disjoint simple closed curves on a framed Riemannian surface $(X,v)\in \mathcal{X}$. Each element $[\gamma_{i}]\in S$ determines a covering $p_{i}: (Y_{i},w_{i}) \rightarrow (X,v)$, where $Y_{i}$ is a Riemann surface with fundamental group $\mathbb{Z}$. We denote
\[
    \theta(X,v,S)=\sum_{i} (p_{i})_{*}\phi(Y_{i}, w_{i})
\]
the quartic differential on $X$ associated to the system of curves $S$. 

\begin{lemma} The differential $\theta(X,v,S)$ is holomorphic with poles of order at most $4$ at the punctures of $X$. Moreover, $\theta\neq 0$, whenever $X$ is closed or $X$ has punctures and $S$ contains a peripheral curve.
\end{lemma}
\begin{proof}If $Y_{i}$ is an annulus, then $\phi(Y_{i},w_{i})$ is integrable and hence its push-forward is integrable. Otherwise $Y_{i}$ is a punctured plane or a punctured disk. Each puncture of $Y_{i}$ has a neighborhood that is mapped injectively to a neighborhood of a puncture in $X$, creating a pole of order $4$ for $\theta$. Since $\phi(Y_{i},w_{i})$ is integrable outside a neighborhood of the punctures, it push-forward its holomorphic. 

It is clear that $\theta$ is not identically zero when $S$ contains a peripheral curve, as $\theta$ has a pole of order $4$ at the corresponding puncture. Therefore, we are only left to consider the case of $X$ closed and $S$ consisting of homotopy classes of disjoint simple closed curves. This will follow from the fact that the quartic differentials $(p_{i})_{*}\phi(Y_{i},w_{i})$ are linearly independent. Let $\rho_{0}:\pi_{1}(X)\rightarrow \Sp(4,\R)$ denote the Fuchsian representation uniformizing $X$. Recall that a Hitchin representation $\rho:\pi_{1}(X) \rightarrow \Sp(4,\R)$ has the property that every simple closed curve is mapped to a diagonalizable matrix with distinct eigenvalues $\lambda_{1}>\lambda_{2}>\lambda_{2}^{-1}>\lambda_{1}^{-1}>0$. The length functions
\[
    \ell_{1}(\rho(\gamma_{i}))=\log\left(\frac{\lambda_{1}(\rho(\gamma_{i}))}{\lambda_{2}(\rho(\gamma_{i}))}\right) \ \ \ \text{and} \ \ \ \ell_{2}(\rho(\gamma_{i}))=\log\left(\frac{\lambda_{2}(\rho(\gamma_{i}))}{\lambda_{3}(\rho(\gamma_{i}))}\right)
\]
for a collection of curves $\gamma_{i}$ in a pants decomposition belong to a coordinate system of the Hitchin component, hence their differentials at $\rho$ are linearly independent. Consider now the function
$m_{i}(\rho)=\ell_{1}(\rho(\gamma_{i}))+\ell_{2}(\rho(\gamma_{i}))$: note that $m_{i}$ vanishes identically at the Fuchsian locus, because for any Fuchsian representation $\lambda_{1}=\lambda^{3}$ and $\lambda_{2}=\lambda$ for some $\lambda>1$. In particular, the differentials $dm_{i}$ at $\rho_{0}$ annihilate any vector tangent to the Fuchsian locus. On the other hand, Labourie and Wentworth proved a generalization of Gardiner’s formula (\cite[Corollary 4.0.5]{LW_Fuchsianlocus}) that relates $dm_{i}$ with the quartic differentials $(p_{i})_{*}\phi(Y_{i},w_{i})$. More precisely, if we identify the tangent space at $\rho_{0}$ with $H^{0}(X,K^{2})\oplus H^{0}(X, K^{4})$, there are constants $c_{1},c_{2}\neq 0$ such that for any $q \in H^{0}(X,K^{4})$ we have
\[
    dm_{i}(q)=c_{1}\Re\left(\int_{X}\bar{q} \cdot (p_{i})_{*}\phi(Y_{i},w_{i})\sigma^{-4} dA_{\sigma}\right)+ c_{2}\Re\left(\int_{X}\bar{q} \cdot (p_{i})_{*}\phi(Y_{i},w_{i})\sigma^{-4} dA_{\sigma}\right) \ ,
\]
where $\sigma$ denotes the conformal hyperbolic metric on $X$. It follows that if $\theta$ vanished identically for some system of curves $\gamma_{i}$, we would have that $\sum_{i}dm_{i}=0$ because $H^{0}(X,K^{4})$ is a supplement of the tangent space at the Fuchsian locus, but this would contradict the fact that $d\ell_{1}(\rho_{0}(\gamma_{i}))$ and $d\ell_{2}(\rho_{0}(\gamma_{i}))$ are linearly independent.

\end{proof}

\subsection{Flat metrics as geodesic currents}
Geodesic currents have been shown to encode the data of a great host of different geometric objects: Bonahon \cite{Bonahon_currents} first demonstrated the Liouville currents captured the marked length spectra of closed hyperbolic surfaces; Otal \cite{Otal} extended the result to include metrics of negative curvature. Work of Duchin, Leininger and Rafi \cite{DLR_flat} shows the flat metrics coming from holomorphic quadratic differentials may be regarded as geodesic currents, and more recent work of Bankovic and Leininger \cite{Rigidity_flat} extends the results to include nonpositively curved Euclidean cone metrics.
Hence to any holomorphic quartic differential metric $|q|^{\frac{1}{2}}$, there is a geodesic current $L_{q}$ with the property that for any closed curve $\gamma,$
$$l_{|q|^{\frac{1}{2}}}([\gamma]) = i(L_{q}, \delta_{\gamma}),$$
where $l_{|q|^{\frac{1}{2}}}([\gamma])$ denotes the length of a geodesic representative of $[\gamma]$ for the flat metric $|q|^{\frac{1}{2}}$. Recall that the geodesic representative may not be unique. This, however, only happens when they foliate a flat cylinder, so all representatives have the same length. 
Moreover, the area formula for geodesic currents coming from negatively curved metrics also holds for holomorphic $n$ differential metrics,
$$i(L_{q}, L_{q}) = \frac{\pi}{2} \Area(S, |q|^{\frac{2}{n}}) \ . $$
In particular, whenever the holomorphic quartic differential metric $|q|^{\frac{1}{2}}$ is of unit area, we have $i(L_{q}, L_{q}) = \pi/2$.

\subsection{Degeneration of flat metrics}
As the quartic differential metrics may be regarded as geodesic currents, we obtain the following natural continuous injective (see \cite{Rigidity_flat}) map
$$\Flat_{4}^{1}(S) \xrightarrow{\widehat{\mathpzc{L}}} \mathcal{C}(S) \xrightarrow{\pi} \mathbb{P}\mathcal{C}(S)\ .$$

To obtain a length spectrum compactification of the unit area quartic differentials metrics, we will take the closure of the image of $\Flat_{4}^{1}(S)$ in $\mathbb{P}\mathcal{C}(S)$. This will extend the results of Duchin, Leininger and Rafi \cite{DLR_flat} for quadratic differentials and the authors \cite{OT} for cubic differentials: the boundary of this compactification contains \textit{mixed structures}, where the flat regions will now come from integrable meromorphic quartic differentials. \\

A \textit{mixed structure} $\eta$ is a geodesic current constructed as follows. Let $S' \subset S$ be a disjoint collection of incompressible subsurfaces each with negative Euler characteristic. Each connected component of $S'$ may be regarded as a punctured surface, upon which one may choose a punctured Riemann surface structure and a meromorphic quartic differential with possible poles of order at most 3 at the punctures. The meromorphic quartic differentials will yield finite area flat metrics, with boundary curves having length zero. For each connected component of the complement of $S'$, choose a measured lamination, whose support may include the boundary curves. If $q$ denotes a meromorphic quartic differential on $S'$, and $\lambda$ the measured lamination supported on the complement of $S'$, then the mixed structure $\eta$ is
$$ \eta = L_{q} + \lambda.$$
It is entirely possible for $S' = \emptyset$ or $S'=S$, in which case $\eta$ is a measured lamination or a quartic differential metric respectively.\\

Let $\Mix(S) \subset \mathcal{C}(S)$ denote the space of mixed structures. If $\pi: \mathcal{C}(S) \to \mathbb{P} \mathcal{C}(S)$ is the projection map, then $\pi(\Mix(S)) = \mathbb{P} \Mix(S)$.

\begin{thm}\label{thm:limit_flat}
The closure of $\Flat_{4}({S})$ is precisely $\mathbb{P}\Mix(S)$.
\end{thm}

\begin{proof}
The proof will follow by showing containment in both directions.  
\subsubsection*{\underline{\textnormal{Part I}}} We begin by showing $\partial \, \overline{\Flat_{4}({S})} \subset \mathbb{P}\Mix(S)$. As the space of projectivized currents is compact, then to any sequence $|q_{n}|^{\frac{1}{2}}$ of unit area flat metrics, (up to a subsequence) there exists a sequence of scalars $t_{n}$ so that $t_{n} L_{q_{n}} \to L_{\infty}$. We observe the following dichotomy: either $t_{n}$ converges to 0, or $t_{n}$ is bounded away from $0$ and infinity. In fact, because every geodesic current has finite self-intersection number, the sequence $t_{n}$ cannot diverge to $+\infty$. In the first case, then 
\begin{align*}
i(L_{\infty}, L_{\infty}) &= \lim_{n \to \infty} i(t_{n} L_{q_{n}}, t_{n} L_{q_{n}}) \\
&= \lim_{n \to \infty} t_{n}^{2} \, i(L_{q_{n}}, L_{q_{n}}) \\
&= 0,
\end{align*}
from which it follows $L_{\infty}$ is a measured lamination. 

If however, $t_{n}$ is bounded away from $0$ and infinity, then without loss of generality, we may take $t_{n}$ to be identically 1. The limit of $L_{q_{n}}$ then will be in the same projective class as $L_{\infty}$. The current decomposition \cite{BIPP} allows us to decompose $L_{\infty}$ as a sum 
$$L_{\infty} = \sum_{W \subset S} \mu_{W} + \sum_{\alpha \in \mathcal{E}} c_{\alpha} \delta_{\alpha}. $$
Here $\mathcal{E}$ is the set
$$\mathcal{E} =\{ \alpha \in C(S) \ | \ i(L_{\infty}, \delta_{\alpha})=0 \ \text{and} \ i(L_{\infty}, \delta_{\beta})> 0 \ \forall \beta \ \text{such that} \ i(\delta_{\alpha}, \delta_{\beta})> 0 \ \},$$ 
where the $W$'s are the connected components of $S \setminus \mathcal{E}$
and $c_{\alpha}$ are nonnegative weights. The first part of the proof is complete provided we show each $\mu_{W}$ is a current coming from a integrable quartic differential, or a measured lamination. 
Recall the \textit{systole} of $L_{\infty}$ on $W$, denoted $\sys_{W}(L_{\infty})$, is the infimum of the set
$$D(W)=\{i(L_{\infty}, \delta_{\gamma}) \ | \ \gamma \in C(W)\}.$$

\noindent There are two cases for $\mu_{W}$:
\begin{enumerate}[i)]
\item if $\sys_{W}(L_{\infty})=0$, then by (\cite[Theorem 1.1]{BIPP}) the restriction of $L_{\infty}$ to $W$ is a measured lamination;
    \item if $\sys_{W}(L_{\infty})>0$, then we have a uniform lower-bound for the $|q_{n}|^{\frac{1}{2}}$-length of any non-peripheral simple closed curve, and hence also of any nonperipheral closed curve in $W$. Since $W$ is a connected component of $S \setminus \mathcal{E}$, the $|q_{n}|^{\frac{1}{2}}$-lengths of the boundary curves in $W$ go to $0$. Therefore, after choosing a basepoint in $W$ (away from the boundary) and passing to a subsequence, we can assume that $q_{n}$ restricted to $W$ converges geometrically to a meromorphic quartic differential on $W$, which must necessarily have poles of order at most $3$ at the punctures since $L_{\infty}$ has finite self-intersection number. This in particular implies convergence of the length spectrum and thus convergence in the space of geodesic currents.
\end{enumerate}
Hence $L_{\infty}$ is a mixed structure, and this concludes the first part of the proof.

\subsubsection*{\underline{\textnormal{Part II}}} Let $\eta=\lambda'+L_{q} \in \Pp\Mix(S)$, where $q$ is a meromorphic quartic differential on an incompressible subsurface $S'\subset S$ and $\lambda'$ is a lamination whose support is disjoint from $S'$. Recall that the simple closed curves $\gamma_{i}$ homotopic to the boundaries of $S'$ may or may not be part of the lamination $\lambda'$, so we will write $\lambda'=\lambda+\sum_{i}w_{i}\delta_{\gamma_{i}}$, for some non-negative weights $w_{i}$. Let $S''=S\setminus S'$. It is clear that we can find a complex structure on $S''$ and a sequence of meromorphic quartic differentials $q_{n}$ such that: (i) the boundary components of $S''$ are conformal to punctures, (ii) the length spectrum of the flat metrics $|q_{n}|^{\frac{1}{2}}$ converges to that of the lamination $\lambda$, and (iii) its area goes to $0$. In fact, there is sequence of holomorphic quadratic differentials with such properties (see \cite{DLR_flat}) and its tensor product gives the desired sequence of quartic differentials.

If $q$ has no poles of order three, we cut a geodesic slit of length $1/n$ starting at each puncture of $S'$ and $S''$. At this step, if the weight $w_{i}\neq 0$, we also insert a flat cylinder of circumference $2/n$ and width $w_{i}$. Note that we can arrange the cuts so that the identifications of the boundaries only involve translations and rotations of $k\frac{\pi}{2}$ for some $k \in \mathbb{Z}$.  Because $q$ has no poles of order three, the flat metrics  we obtain in this way have cone angles $2\pi +k\frac{\pi}{2}$ with $k\geq 0$, and they are induced by holomorphic quartic differentials $q_{n}$. By construction $L_{q_{n}}$ converges to $\eta$.

If $q$ has a pole of order three, a general procedure described in \cite[Section 7]{strata} allows us to break up a triple pole into two arbitrarily close poles of order $1$ and $2$ by modifying the metric $|q|^{\frac{2}{3}}$ only in a neighborhood of the puncture (see also \cite[Section 3]{Rafe_flat}). We can then apply the same surgery described before, where now the slit is given by a geodesic segment connecting the two new-born singularities.
\end{proof}

\section{Limits of induced metrics}
In this section we study how the length spectrum of induced metrics on maximal surfaces in $\h^{2,2}$ degenerate along diverging sequences of  $\SO_{0}(2,3)$-Hitchin representations, thus giving a proof of Theorem \ref{thmB}. 

\subsection{The space of induced metrics} 
We start by collecting some properties of the space of induced metrics $\Ind(S)$ that will be useful in the sequel.

\begin{prop}
There is a bijection between the space $\Ind(S)$ of induced metrics and $\mathcal{Q}^{4}/S^{1}$, where $S^{1}$ acts by multiplication on the fibers.
\end{prop}

\begin{proof}
As we saw in Section \ref{sec:background}, any particular induced metric may be written as $4e^{\psi_{1}-\psi_{2}}\sigma$, where $2\sigma$ is the Riemannian metric of constant curvature $-1$ in the conformal class of the maximal surface, and $\psi_{1}$ and $\psi_{2}$ satisfy the system
\begin{equation}\label{eq:PDE}
\begin{cases}
\Delta_{\sigma} \psi_{1} =e^{\psi_{1} - \psi_{2}}-e^{-2\psi_{1}} \|q\|^{2}_{\sigma} +\frac{3}{4}\kappa(\sigma) \\ 
\Delta_{\sigma} \psi_{2} =e^{2\psi_{2}}- e^{\psi_{1} - \psi_{2}} + \frac{1}{4}\kappa(\sigma)
\end{cases} \ .
\end{equation}
The proposition amounts to showing that if $\tilde{q}$ is not in the $S^{1}$-orbit of $q$, then  $e^{\psi_{1}-\psi_{2}} \sigma\neq e^{\widetilde{\psi}_{1}-\widetilde{\psi}_{2}}\sigma$. Assume not, then $\psi_{1} - \psi_{2} = \widetilde{\psi}_{1} - \widetilde{\psi}_{2}$.
Consider the equation
$$ \Delta_{\sigma}(\psi_{2} - \widetilde{\psi}_{2}) = e^{\widetilde{\psi}_{1} - \widetilde{\psi}_{2}}-e^{\psi_{1} - \psi_{2}}+e^{2\psi_{2}}-e^{2 \widetilde{\psi}_{2}}=e^{2\psi_{2}}-e^{2 \widetilde{\psi}_{2}}.  $$
At the minimum of $\psi_{2}-\widetilde{\psi}_{2}$, we have
\[
    0\leq \Delta_{\sigma}(\psi_{2}-\widetilde{\psi}_{2})=e^{2\psi_{2}}-e^{2\widetilde{\psi}_{2}}
\]
which implies $\psi_{2}-\widetilde{\psi}_{2}\geq0$ everywhere on $S$. On the other hand, at the maximum of $\psi_{2}-\widetilde{\psi}_{2}$, we have 
$$ 0 \geq \Delta_{\sigma} (\psi_{2} - \widetilde{\psi}_{2}) = e^{2 \psi_{2}} - e^{2 \widetilde{\psi}_{2}}$$
yielding $\psi_{2}-\widetilde{\psi}_{2}\leq 0$ everywhere on $S$. We deduce that $\psi_{2}=\widetilde{\psi}_{2}$ at every point, and, consequently, $\psi_{1}=\widetilde{\psi}_{1}$ as well. Consider, now, the equation
\begin{align*}
0 &= \Delta_{\sigma} (\psi_{1} - \psi_{2}) - \Delta_{\sigma} (\widetilde{\psi}_{1} - \widetilde{\psi}_{2})\\
&= \left(2e^{\psi_{1} - \psi_{2}} - e^{2\psi_{2}} - e^{-2\psi_{1}} \|q\|_{\sigma}^{2} \right) - \left(2e^{\widetilde{\psi}_{1} - \widetilde{\psi}_{2}} - e^{2 \widetilde{\psi}_{2}} - e^{-2 \widetilde{\psi}_{1}} \|\widetilde{q}\|_{\sigma}^{2}\right) \\
&= \left(2e^{\psi_{1} - \psi_{2}} - 2e^{\widetilde{\psi}_{1} - \widetilde{\psi}_{2}} \right) - \left( e^{2\psi_{2}} -  e^{2 \widetilde{\psi}_{2}} \right) - \left(e^{-2\psi_{1}} \|q\|_{\sigma}^{2} -   e^{-2 \widetilde{\psi}_{1}} \|\widetilde{q}\|_{\sigma}^{2}   \right) \\
&=e^{-2\psi_{1}}\left(
 \|q\|_{\sigma}^{2} -\|\widetilde{q}\|_{\sigma}^{2}\right) \ .
\end{align*}
We find that $|q|=|\tilde{q}|$ at every point, which easily implies that $q$ and $\tilde{q}$ lie in the same $S^{1}$-orbit, thus obtaining a contradiction.
\end{proof}

\begin{rmk} We will show in Proposition \ref{prop:par_Ind} that the map 
\begin{align*}
    \mathcal{G}: \mathcal{Q}^{4}/S^{1} &\rightarrow \Ind(S) \\
    (\sigma, [q]) & \mapsto 4e^{\psi_{1}-\psi_{2}}\sigma \ ,
\end{align*}
where $\psi_{1}$ and $\psi_{2}$ are solutions to Hitchin equations (\ref{eq:PDE}), is continuous and proper, thus giving a parametrization of $\Ind(S)$.
\end{rmk}

A key ingredient for our construction is that all metrics in $\Ind(S)$ are negatively curved. This fact has been recently proven in (\cite{LTW}), but we provide here a different proof, communicated to us by Qiongling Li, based only on the analysis of the Hitchin equations.

\begin{prop}\label{prop:neg_curv} Let $\psi_{1}$ and $\psi_{2}$ be the solutions to Equations (\ref{eq:PDE}) with given data $(\sigma, q)\in \mathcal{Q}^{4}(S)$. Then the metric $4e^{\psi_{1}-\psi_{2}}\sigma$ is strictly negatively curved.
\end{prop}
\begin{proof} It is sufficient to show that $h=e^{\psi_{1}-\psi_{2}}\sigma$ is negatively curved. The curvature of $h$ is
\[
    \kappa(h)=-2\Delta_{h}\log(h)=2\left(\frac{e^{-2\psi_{1}}\|q\|_{\sigma}^{2}}{e^{\psi_{1}-\psi_{2}}}+\frac{e^{2\psi_{2}}}{e^{\psi_{1}-\psi_{2}}}-2\right),
\]
hence it is strictly negative if and only if $f_{1}+f_{2}<2$, where 
\[
    f_{1}=\frac{e^{-2\psi_{1}}\|q\|_{\sigma}^{2}}{e^{\psi_{1}-\psi_{2}}} \ \ \ \ \text{and} \ \ \ \ f_{2}=\frac{e^{2\psi_{2}}}{e^{\psi_{1}-\psi_{2}}} \ .
\]
Now, outside the zeros of $q$ we have that
\begin{align*}
    \Delta_{h}\log(f_{1})&=-3\Delta_{h}\psi_{1}+\Delta_{h}\psi_{2}-4\Delta_{h}\log(\sigma)=3\frac{e^{-2\psi_{1}}\|q\|_{\sigma}^{2}}{e^{\psi_{1}-\psi_{2}}}+\frac{e^{2\psi_{2}}}{e^{\psi_{1}-\psi_{2}}}-4=-4+3f_{1}+f_{2} \\
    \Delta_{h}\log(f_{2})&=3\Delta_{h}\psi_{2}-\Delta_{h}\psi_{1}=3\frac{e^{2\psi_{2}}}{e^{\psi_{1}-\psi_{2}}}+\frac{e^{-2\psi_{1}}\|q\|_{\sigma}^{2}}{e^{\psi_{1}-\psi_{2}}}-2=3f_{2}+f_{1}-2 \ .
\end{align*}
Then
\begin{align*}
    \Delta_{h}\log(f_{1}+f_{2})&\geq \frac{f_{1}\Delta_{h}\log(f_{1})+f_{2}\Delta_{h}\log(f_{2})}{f_{1}+f_{2}}=\frac{-4f_{1}+3f_{1}^{2}+2f_{1}f_{2}+3f_{2}^{2}+f_{1}f_{2}-2f_{2}}{f_{1}+f_{2}}\\
    &>\frac{-4f_{1}-4f_{2}+3f_{1}^{2}+2f_{1}f_{2}+3f_{2}^{2}+f_{1}f_{2}}{f_{1}+f_{2}}\geq -4+2(f_{1}+f_{2}) \ .
\end{align*}
We deduce that, if $f_{1}+f_{2}$ takes its maximum outside the zeros of $q$, then $f_{1}+f_{2}<2$ and $\kappa_{g}<0$. On the other hand, if $f_{1}+f_{2}$ takes its maximum at a zero $p$ of $q$, then $f_{1}(p)=0$ and $p$ is a maximum of $f_{2}$. But then,
\[
    0\geq \Delta_{h}\log(f_{2})(p)=-2+3f_{2}(p)+f_{1}(p)=-2+3f_{2}(p)
\]
which implies that $f_{2}(p)\leq \frac{2}{3}$ and $f_{1}+f_{2}\leq f_{1}(p)+f_{2}(p)\leq \frac{2}{3}$ everywhere on $S$. Hence $\kappa_{g}<0$ in this case as well.  
\end{proof}

It will also be useful to compare the induced metric $g=4e^{\psi_{1}-\psi_{2}}\sigma$ with the constant curvature metric $\sigma$. To this aim we construct constant sub- and super-solutions for the system of differential equations (\ref{eq:PDE}). We refer the reader to \cite{Amann} and \cite{TW} for an overview of the sub- and super-solution theory for systems of PDEs. 

\begin{lemma}\label{lm:subsolution} The pair $(u_{1},u_{2})=\left(\frac{3}{2}\log\left(-\frac{3}{4}\kappa(\sigma)\right),\frac{1}{2}\log\left(-\frac{3}{4}\kappa(\sigma)\right)\right)$ is a sub-solution for Equation (\ref{eq:PDE}).
\end{lemma}
\begin{proof} Let $F_{1}(\psi_{1},\psi_{2})$ and $F_{2}(\psi_{1},\psi_{2})$ denote the right-hand side of Equation (\ref{eq:PDE}). We need to check that $F_{i}(u_{1},u_{2}) \leq 0$ for $i=1,2$. This follows from straightforward computations:
\[
    F_{1}(u_{1},u_{2})=e^{u_{1}-u_{2}}-e^{-2u_{1}}\|q\|_{\sigma}^{2}+\frac{3}{4}\kappa(\sigma) \leq e^{u_{1}-u_{2}}+\frac{3}{4}\kappa(\sigma)=0
\]
and
\[
    F_{2}(u_{1},u_{2})=e^{2u_{2}}-e^{u_{1}-u_{2}}+\frac{1}{4}\kappa(\sigma)=\frac{1}{4}\kappa(\sigma)\leq 0 \ .
\]
\end{proof}

\begin{lemma}\label{lm:supersolution} There are constants $c_{1},c_{2}>0$ such that the pair $(c_{1},c_{2})$ is a super-solution of Equation (\ref{eq:PDE}).
\end{lemma}
\begin{proof} Using the same notation of the previous lemma, we have to find constants $c_{1},c_{2}>0$ such that
\begin{equation}\label{eq:ineq}
    \begin{cases}
        e^{c_{1}-c_{2}}-e^{-2c_{1}}\|q\|_{\sigma}^{2}+\frac{3}{4}\kappa(\sigma)\geq 0 \\
        e^{2c_{2}}-e^{c_{1}-c_{2}}+\frac{1}{4}\kappa(\sigma) \geq 0
    \end{cases}
\end{equation}
Dividing the second equation by $e^{2c_{2}}$, we get
\begin{equation}
    1-e^{c_{1}-3c_{2}}+\frac{1}{4}\kappa(\sigma)e^{-2c_{2}} \geq 1-e^{c_{1}-3c_{2}}+\frac{1}{4}\kappa(\sigma)
\end{equation}
where in the last inequality we used that $c_{2}>0$ and $\kappa(\sigma)<0$. It is easy to check that if $c_{1}=3c_{2}+\log\left(1+\frac{1}{4}\kappa(\sigma)\right)$ then 
\[
    1-e^{c_{1}-3c_{2}}+\frac{1}{4}\kappa(\sigma)=0,
\]
hence the second inequality in (\ref{eq:ineq}) holds. We determine $c_{2}$ by looking now at the first equation in (\ref{eq:ineq}):
\begin{align*}
   &e^{2c_{2}}\left(1+\frac{1}{4}\kappa(\sigma)\right)-e^{-6c_{2}}\left(1+\frac{1}{4}\kappa(\sigma)\right)^{-2}\|q\|_{\sigma}^{2}+\frac{3}{4}\kappa(\sigma) \\
   &\geq  e^{2c_{2}}\left(1+\frac{1}{4}\kappa(\sigma)\right)-e^{-6c_{2}}\left(1+\frac{1}{4}\kappa(\sigma)\right)^{-2}\max_{S}(\|q\|_{\sigma}^{2})+\frac{3}{4}\kappa(\sigma)
\end{align*}
and this last expression becomes positive as long as $c_{2}$ is sufficiently large. Note that we can 
arrange $c_{2}$ to only depend on $\sigma$ and $\max_{S}(\|q\|_{\sigma}^{2})$. 
\end{proof}

Moreover, the induced metric $g=4e^{\psi_{1}-\psi_{2}}\sigma$ dominates a metric of constant curvature in its conformal class:
\begin{lemma}\label{lm:domination_constant} Given $(\sigma,q)\in \mathcal{Q}^{4}(S)$, let $g=4e^{\psi_{1}-\psi_{2}}\sigma$ be the induced metric on the maximal surface. Then $-3\kappa(\sigma)\sigma \leq g$. 
\end{lemma}
\begin{proof} Let $f=\psi_{1}-\psi_{2}$ and $g=2\psi_{2}$, then
\begin{align*}
    \Delta_{\sigma}f&=\Delta_{\sigma}\psi_{1}-\Delta_{\sigma}\psi_{2}=2e^{\psi_{1}-\psi_{2}}-e^{-2\psi_{1}}\|q\|^{2}_{\sigma}-e^{2\psi_{2}}+\frac{1}{2}\kappa(\sigma) \leq 2e^{f}-e^{g}+\frac{1}{2}\kappa(\sigma) \\
    \Delta_{\sigma}g&=2\Delta_{\sigma}\psi_{2}=2e^{2\psi_{2}}-2e^{\psi_{1}-\psi_{2}}+\frac{1}{2}\kappa(\sigma) = 2e^{g}-2e^{f}+\frac{1}{2}\kappa(\sigma)
\end{align*}
Let $x$ be a point of minimum for $f$ and $y$ be a point of minimum for $g$. From the above equations, we find that
\begin{align*}
    2e^{f} &\geq 2e^{f(x)}\geq e^{g(x)}-\frac{1}{2}\kappa(\sigma) \\
    2e^{g}&\geq 2e^{g(y)}\geq 2e^{f(y)}-\frac{1}{2}\kappa(\sigma) \ .
\end{align*}
Now, 
\[
    e^{g(x)} \geq e^{f(y)}-\frac{1}{4}\kappa(\sigma)\geq \frac{1}{2}e^{g(x)}-\frac{1}{2}\kappa(\sigma)
\]
thus $e^{g(x)}\geq -\kappa(\sigma)$ and 
\[
    g=4e^{\psi_{1}-\psi_{2}}\sigma=4e^{f}\sigma \geq (2e^{g(x)}-\kappa(\sigma))\sigma \geq -3\kappa(\sigma)\sigma
\]
as claimed. 
\end{proof}

\subsection{Comparison with the quartic differential metric} The main step towards the proof of Theorem \ref{thmB} is a comparison between the induced metric $\mathcal{G}(\sigma, [q])$ and the flat metric $|q|^{\frac{1}{2}}$. \\

Let us start with the construction of sub-solutions for the system (\ref{eq:PDE}) in terms of the flat metric $|q|^{\frac{1}{2}}$.

\begin{prop}\label{prop:sub_flat} Let $\psi_{1}$ and $\psi_{2}$ be the solution to Equation (\ref{eq:PDE}). Then $\psi_{1}\geq \log(\|q\|_{\sigma}^{\frac{3}{4}})$ and $\psi_{2}\geq \log(\|q\|_{\sigma}^{\frac{1}{4}})$. 
\end{prop}
\begin{proof} We define the following functions
\[
    v_{1}=\begin{cases}
            \frac{3}{2}\log(\|q\|_{\sigma}^{\frac{1}{2}})  \ \ \ \ \ \ \ \ \ \ \ \ \ \ \ \ \ \ \ \ \ \text{if} \ \ \|q\|_{\sigma}^{\frac{1}{2}} \geq -\frac{3}{4}\kappa(\sigma) \\
            \frac{3}{2}\log\left(-\frac{3}{4}\kappa(\sigma)\right) \ \ \ \ \ \ \ \ \ \ \ \ \ \ \ \ \text{otherwise}
          \end{cases}
\]
and 
\[
    v_{2}=\begin{cases}
            \frac{1}{2}\log(\|q\|_{\sigma}^{\frac{1}{2}})  \ \ \ \ \ \ \ \ \ \ \ \ \ \ \ \ \ \ \ \ \ \text{if} \ \ \|q\|_{\sigma}^{\frac{1}{2}} \geq -\frac{3}{4}\kappa(\sigma) \\
            \frac{1}{2}\log\left(-\frac{3}{4}\kappa(\sigma)\right) \ \ \ \ \ \ \ \ \ \ \ \ \ \ \ \ \text{otherwise}
          \end{cases}
\]
The proposition follows if we show that the pair $(v_{1},v_{2})$ is a sub-solution for the system (\ref{eq:PDE}). As in Lemma \ref{lm:subsolution}, we denote by $F_{i}(\psi_{1}, \psi_{2})$ the right-hand side of Equation (\ref{eq:PDE}). Because $v_{1}$ and $v_{2}$ are defined as the maximum of two functions and the maximum of two sub-solutions is a sub-solution, it is sufficient to show that 
\[
    \begin{cases}
    \Delta_{\sigma}\log(\|q\|_{\sigma}^{\frac{3}{4}}) \geq F_{1}(\log(\|q\|_{\sigma}^{\frac{3}{4}}),\log(\|q\|_{\sigma}^{\frac{1}{4}})) \\
    \Delta_{\sigma}\log(\|q\|_{\sigma}^{\frac{1}{4}}) \geq F_{2}(\log(\|q\|_{\sigma}^{\frac{3}{4}}),\log(\|q\|_{\sigma}^{\frac{1}{4}}))
    \end{cases} \ 
\]
outside the zeros of $q$ and
\[
    \begin{cases}
    F_{1}\left(\frac{3}{2}\log\left(-\frac{3}{4}\kappa(\sigma)\right), \frac{1}{2}\log\left(-\frac{3}{4}\kappa(\sigma)\right)\right) \leq 0 \\
    F_{2}\left(\frac{3}{2}\log\left(-\frac{3}{4}\kappa(\sigma)\right), \frac{1}{2}\log\left(-\frac{3}{4}\kappa(\sigma)\right)\right) \leq 0
    \end{cases} \ .
\]
The first computation is straightforward and left to the reader. The second was carried out in Lemma \ref{lm:subsolution}.
\end{proof}

This result can be improved to a domination result between the induced metric $g$ and the flat metric $4|q|^{\frac{1}{2}}$, which already appeared in
\cite{QL_cyclic}. We provide however a proof for the convenience of the reader.

\begin{prop}\label{prop:sub_ineq} Given $(\sigma, q) \in \mathcal{Q}^{4}(S)$, let $g=4e^{\psi_{1}-\psi_{2}}\sigma$ be the induced metric on the maximal surface. Then, $4|q|^{\frac{1}{2}} \leq g$. 
\end{prop}
\begin{proof} We consider the function $u=e^{4\psi_{2}-4\psi_{1}}\|q\|_{\sigma}^{2}$, which is well-defined everywhere on $S$. It is clear that $u\geq 0$ and takes a maximum outside the zeros of $q$. Now, outside the zeros of $q$, the function $u$ satisfies
\begin{align*}
    \Delta_{\sigma} \log(u) &= 4\Delta_{\sigma} \psi_{2}-4\Delta_{\sigma} \psi_{1}-4\Delta_{\sigma}\log(\sigma) \\
    &= 4e^{2\psi_{2}}-8e^{\psi_{1}-\psi_{2}}+4e^{-2\psi_{1}}\|q\|_{\sigma}^{2} \\
    &\geq 8e^{\psi_{2}-\psi_{1}}\|q\|_{\sigma}-8e^{\psi_{1}-\psi_{2}} \\
    &= 8e^{\psi_{1}-\psi_{2}}(e^{2\psi_{2}-2\psi_{1}}\|q\|_{\sigma}-1)\\
    &= 8e^{\psi_{1}-\psi_{2}}(\sqrt{u}-1) .
\end{align*}
Hence, at a maximum of $u$, we have
\[
    0\geq \Delta_{g}\log(u) \geq 2(\sqrt{u}-1) \ .
\]
We deduce that $u\leq 1$ which implies that
\[
    4|q|^{\frac{1}{2}}=4\|q\|_{\sigma}^{\frac{1}{2}}\sigma \leq 4e^{\psi_{1}-\psi_{2}} \sigma=g \ .
\]
\end{proof}

We will also need the following uniform bound:

\begin{lemma}[\cite{QL_cyclic}]\label{lm:bound} Let $\psi_{1}$ and $\psi_{2}$ be the solution to Hitchin equations (\ref{eq:PDE}). Then
\[
    0\leq 3\psi_{2}-\psi_{1} \leq \log\left(\frac{4}{3}\right) \ .
\]
\end{lemma}
\begin{proof} Let $f=3\psi_{2}-\psi_{1}$. We compute
\begin{align*}
    \Delta_{\sigma} f&= 3\Delta_{\sigma}\psi_{2}-\Delta_{\sigma}\psi_{1} \\
    &=3e^{2\psi_{2}}-4e^{\psi_{1}-\psi_{2}}+e^{-2\psi_{1}}\|q\|_{\sigma}^{2} \\
    &\geq 3e^{2\psi_{2}}-4e^{\psi_{1}-\psi_{2}} \\
    &=e^{\psi_{1}-\psi_{2}}(3e^{f}-4) \ .
\end{align*}
By the maximum principle, 
\[
    e^{3\psi_{2}-\psi_{1}}=e^{f} \leq \frac{4}{3} 
\]
as claimed. 
\end{proof}

Together with Proposition \ref{prop:neg_curv}, the above estimates allow to study the coarse behavior of the area of the maximal surface in terms of the area of the flat metric $4|q|^{\frac{1}{2}}$:

\begin{prop}\label{prop:area_coarse} For any open region $R \subseteq S$, we have
\[
    \Area(R,4|q|^{\frac{1}{2}}) \leq \Area(R,g) \leq \frac{3}{2}\Area(R,4|q|^{\frac{1}{2}})+6\pi|\chi(S)| \ .
\]
\end{prop}
\begin{proof} The first inequality immediately follows from Proposition \ref{prop:sub_ineq}.\\
As for the second inequality, we start from the formula for the curvature of $g$ found in Proposition \ref{prop:neg_curv}:
\[
    \kappa(g)=\frac{e^{-2\psi_{1}}\|q\|_{\sigma}^{2}}{2e^{\psi_{1}-\psi_{2}}}+\frac{e^{2\psi_{2}}}{2e^{\psi_{1}-\psi_{2}}}-1 \ .
\]
Integrating over $R$, we find that
\begin{align*}
    \Area(R,g)&=\int_{R}\frac{e^{-2\psi_{1}}\|q\|_{\sigma}^{2}}{2e^{\psi_{1}-\psi_{2}}} dA_{g}+ \int_{R}\frac{e^{2\psi_{2}}}{2e^{\psi_{1}-\psi_{2}}} dA_{g} +\int_{R} -\kappa(g) dA_{g} \\
    &\leq 2\int_{R}e^{-2\psi_{1}}\|q\|_{\sigma}^{2} dA_{\sigma} + \frac{1}{2}\int_{R} e^{3\psi_{2}-\psi_{1}} dA_{g} + \int_{S} -\kappa(g) dA_{g} \\
    &= 2\int_{R}e^{-2\psi_{1}}\|q\|_{\sigma}^{2} dA_{\sigma} + \frac{1}{2}\int_{R} e^{3\psi_{2}-\psi_{1}} dA_{g} + 2\pi|\chi(S)| \ .
\end{align*}
Now, by Proposition \ref{prop:sub_flat}, we have
\[
    \int_{R} e^{-2\psi_{1}}\|q\|_{\sigma}^{2} dA_{\sigma} \leq \int_{R} \|q\|_{\sigma}^{-\frac{3}{2}}\|q\|_{\sigma}^{2} dA_{\sigma}=\Area(R, |q|^{\frac{1}{2}})
\]
and by Lemma \ref{lm:bound}, we can estimate
\[
    \int_{R} e^{3\psi_{2}-\psi_{1}} dA_{g} \leq \frac{4}{3}\Area(R,g) \ .
\]
Combining all, we get
\[
    \Area(R,g)\leq \frac{1}{2}\Area(R, 4|q|^{\frac{1}{2}})+\frac{2}{3}\Area(R,g)+2\pi|\chi(S)| \ ,
\]
which gives the desired upper-bound.
\end{proof}

We can finally show as promised that the space of induced metrics $\Ind(S)$ is parametrized by $\mathcal{Q}^{4}/S^{1}$.
\begin{prop}\label{prop:par_Ind} The map 
\begin{align*}
    \mathcal{G}: \mathcal{Q}^{4}/S^{1} &\rightarrow \Ind(S) \\
    (\sigma, [q]) & \mapsto 4e^{\psi_{1}-\psi_{2}}\sigma 
\end{align*}
is a homeomorphism.
\end{prop}
\begin{proof} We already know that $\mathcal{G}$ is bijective, so we only need to show that $\mathcal{G}$ is continuous and proper. Let $(\sigma_{n}, [q_{n}])$ be a sequence converging to $(\sigma, [q])$ in $\mathcal{Q}^{4}/S^{1}$. Let $\psi_{i,n}$ and $\psi_{i}$ for $i=1,2$ denote the solutions to Equation (\ref{eq:PDE}) with data $(\sigma_{n}, [q_{n}])$ and $(\sigma, [q])$, respectively. By Lemmas \ref{lm:subsolution} and \ref{lm:supersolution}, there is a uniform bound on $\Delta_{\sigma_{n}}\psi_{i,n}$, and hence a uniform $C^{1,\alpha}$ bound on $\psi_{i,n}$. We deduce that $\psi_{1,n}$ and $\psi_{2,n}$ converge to weak solutions of Equation (\ref{eq:PDE}) with data $(\sigma, [q])$, which, by uniqueness, consist of the pair $(\psi_{1},\psi_{2})$. By elliptic regularity, the convergence is actually smooth, and $\mathcal{G}$ is continuous.\\
\indent As for properness, assume that  $g_{n}=\mathcal{G}(\sigma_{n},[q_{n}])$ converges smoothly to $g=\mathcal{G}([\sigma,[q])$. It is clear that $\sigma_{n}$ converges to $\sigma$ because those are the hyperbolic metrics in the conformal class of $g_{n}$ and $g$, which converge. By Proposition \ref{prop:area_coarse}, the quantities $\Area(S,4|q_{n}|^{\frac{1}{2}})$ are uniformly bounded. Because $q_{n}$ are holomorphic differentials with respect to uniformly thick complex structures, the sequence $q_{n}$ is lies in a compact set, and thus converges, up to subsequences, to some $q'$. By continuity of $\mathcal{G}$, we find that $\mathcal{G}(\sigma_{n}, [q_{n}])$ converges to $\mathcal{G}(\sigma, [q'])$ and by injectivity, we have $[q']=[q]$. Hence, the sequence $(\sigma_{n}, [q_{n}])$ converges to $(\sigma, [q])$. 
\end{proof}

\subsection{Comparison at high energy}
We now improve some of the estimates of the previous section under the assumption that the area $\|q\|=\Area(S,|q|^{\frac{1}{2}})$ of the quartic differential metric is large. \\

We define the functions
\[
    u_{1}=\psi_{1}-\log(\|q\|_{\sigma}^{\frac{3}{4}}) \ \ \ \text{and} \ \ \ u_{2}=\psi_{2}-\log(\|q\|_{\sigma}^{\frac{1}{4}}) \ ,
\]
which represent the errors between the solutions to Equations (\ref{eq:PDE}) and the subsolutions found in Proposition \ref{prop:sub_flat}. Note that $u_{1}$ and $u_{2}$ are well-defined outside the zeros of $q$ and positive. Moreover, the induced metric can be rewritten as $g=4e^{u_{1}-u_{2}}|q|^{\frac{1}{2}}$. Because $g$ is negatively curved, the function $u_{1}-u_{2}$ is subharmonic and we have the following estimates outside the zeros of $q$:

\begin{lemma}\label{lm:upperbound}
Let $p \in S$ be at distance at least $r>0$ from the zeros of $q$. Then
\[
    (u_{1}-u_{2})(p) \leq \log\left( \frac{\Area(S,g)}{4\pi r^{2}}\right) \ .
\]
In particular, if $r_{0}$ denotes the distance for the renormalized metric $|q|^{\frac{1}{2}}/\|q\|$ of unit area, we have
\[
    (u_{1}-u_{2})(p) \leq \log\left( \frac{\Area(S,g)}{4\pi \|q\|r_{0}^{2}}\right) \ .
\]
\end{lemma}
\begin{proof} Let $B$ be the ball centered at $p$ and radius $r$ for the flat metric $|q|^{\frac{1}{2}}$ that does not contain any zeros of $q$. By Jensen's inequality, we have
\[
    e^{(u_{1}-u_{2})(p)} \leq e^{\fint_{B} u_{1}-u_{2} dA_{q}} \leq \fint_{B} e^{u_{1}-u_{2}}dA_{q}= \frac{1}{4\pi r^{2}}\int_{B} dA_{g} \leq \frac{\Area(S,g)}{4\pi r^{2}} \ ,
\]
and the first estimate follows.\\
The second statement is a simple reformulation, using the fact that $r^{2}=r_{0}^{2}\|q\|$.
\end{proof}

In addition, because $3u_{2}-u_{1}=3\psi_{2}-\psi_{1}$ outside the zeros of $q$, combining the previous estimate with Lemma \ref{lm:bound}, we obtain a bound for $u_{1}+u_{2}$ in terms of the distance from the zeros:
\begin{equation}\label{eq:bound_error}
    u_{1}+u_{2}=(3u_{2}-u_{1})+2(u_{1}-u_{2})\leq \log\left(\frac{4}{3}\right)+2\log\left( \frac{\Area(S,g)}{4\pi \|q\|r_{0}^{2}}\right) \ .
\end{equation}
Note that, for $r_{0}$ fixed, the last term is uniformly bounded as the area of the quartic differential metric tends to infinity thanks to Proposition \ref{prop:area_coarse}. 

\begin{lemma}\label{lm:decay} Let $p$ be a point at distance at least $r_{0}>0$ from the zeros of $q$ with respect to the unit area metric $|q|^{\frac{1}{2}}/\|q\|$. Then, there are constants $c,C>0$, only depending on $r_{0}$, such that
\[
    (u_{1}+u_{2})(p)\leq Ce^{-c\|q\|^{\frac{1}{2}}r_{0}} \ ,
\]
for $\|q\|$ sufficiently large.
\end{lemma}
\begin{proof} Let $B$ be the ball centered at $p$ of radius $r=\|q\|^{\frac{1}{2}}r_{0}$ which does not contain any zeros of $q$. On this ball, the function $u_{1}+u_{2}$ satisfies
\begin{align*}
    \Delta_{\sigma}(u_{1}+u_{2})&=\Delta_{\sigma}\psi_{1}+\Delta_{\sigma}\psi_{2}-2\Delta_{\sigma}\log(\sigma)\\
    &=e^{2\psi_{2}}-e^{-2\psi_{1}}\|q\|_{\sigma}^{2}
    =(e^{2u_{2}}-e^{-2u_{1}})\|q\|_{\sigma}^{\frac{1}{2}} \ ,
\end{align*}
which, choosing a coordinate $z$ so that $q=dz^{4}$ and $z(p)=0$, can be written as
\begin{align*}
    \bar{\partial}\partial(u_{1}+u_{2})=\Delta_{q}(u_{1}+u_{2}) =(e^{2u_{2}}-e^{-2u_{1}}) \ .
\end{align*}
Note that from Equation \ref{eq:bound_error}, there is a constant $a>0$, only depending on $r_{0}$ for $\|q\|$ large, such that
\[
    \bar{\partial}\partial(u_{1}+u_{2})=e^{-2u_{1}}(e^{2u_{2}+2u_{1}}-1)\geq 2a(u_{1}+u_{2}) \ .
\]
Moreover, again by Equation \ref{eq:bound_error}, the function $u_{1}+u_{2}$ is uniformly bounded by a constant $C>0$ on the boundary of $B$ that we can arrange to only depend on $r_{0}$ for $\|q\|$ large. By the maximum principle, we deduce that $u_{1}+u_{2}$ is bounded above by the radial solution of the Dirichlet problem
\[
    \begin{cases}
        \bar{\partial}{\partial}\eta=2a\eta \\
        \eta_{|_{\partial B}}=C
    \end{cases} \ .
\]
It is well-known that $\eta$ can be expressed in terms of the modified Bessel function of the first kind
\[
    \eta(\rho)=C\frac{I_{0}(2\sqrt{2a}\rho)}{I_{0}(2\sqrt{2a}r)}=C\frac{I_{0}(2\sqrt{2a}\rho)}{I_{0}(2\sqrt{2a}\|q\|^{\frac{1}{2}}r_{0})}\ .
\]
Recalling that
\[
    I_{0}(x)\sim \frac{e^{x}}{(2\pi x)^{\frac{1}{2}}} \ \ \ \text{as \ $x \to +\infty$} \ ,
\]
we find 
\[
    (u_{1}+u_{2})(p)\leq\eta(0) \leq Ce^{-c\|q\|^{\frac{1}{2}}r_{0}}
\]
for some $c>0$ and $\|q\|$ sufficiently large.
\end{proof}

\begin{cor}\label{cor:decay} Let $q_{n}$ be a sequence of holomorphic quartic differentials such that $\|q_{n}\| \to +\infty$ and $|q_{n}|^{\frac{1}{2}}/\|q\|$ converges geometrically to a flat metric induced by a meromorphic quartic differential $q'$ on a subsurface $S'\subset S$. Fix $\epsilon>0$. Then there is $n_{0} \in \mathbb{N}$ such that for all $n\geq n_{0}$ and all $p\in S'$ at $q'$-distance at least $2\epsilon$ from the zeros and poles of $q'$ we have
\[
    (u_{1,n}+u_{2,n})(p) \leq Ce^{-c\|q\|^{\frac{1}{2}}\epsilon} \ .
\]
In particular, $u_{1}$ and $u_{2}$ decay exponentially to $0$ outside the zeros and poles of $q'$.
\end{cor}
\begin{proof} By assumption the normalized metrics $|q_{n}|^{\frac{1}{2}}/\|q\|$ converge uniformly on compact sets of $S'$ outside the cone singularities of $q'$. Therefore, for $n$ sufficiently large, the zeros of $q_{n}$ are all contained in balls of $q'$-radius $\epsilon/2$ centered at the singularities of $|q'|^{\frac{2}{3}}$. We deduce that there is $n_{0} \in \mathbb{N}$ such that for all $n\geq n_{0}$ the following conditions hold: (i) all $p\in S'$ at $q'$-distance at least $2\epsilon$ from the zeros and poles of $q'$ are at distance at least $\epsilon$ from the zeros of $q_{n}$ with respect to the unit area metrics $|q_{n}|^{\frac{1}{2}}/\|q\|$; (ii) the area $\|q_{n}\|$ is sufficiently large so that the estimate in Lemma \ref{lm:bound} holds with $r_{0}=\epsilon$. This gives the desired decay for $u_{1,n}$ and $u_{2,n}$ at $p$ because they are positive, hence smaller than $u_{1,n}+u_{2,n}$. 
\end{proof}

\begin{cor}\label{cor:conformal_factor} Under the assumptions of Corollary \ref{cor:decay}, the conformal factor of the induced metrics $g_{n}$ with respect to the flat metrics $4|q_{n}|$ converges to $1$ uniformly on compact sets outside the cone singularities of $q'$.
\end{cor}
\begin{proof} This follows immediately from the fact that $g_{n}=4e^{u_{1,n}-u_{2,n}}|q_{n}|^{\frac{1}{2}}$ and the decay of $u_{1,n}$ and $u_{2,n}$ in the previous corollary.
\end{proof}

We now use this asymptotic decay to improve the estimate on the behavior of the area of the induced metric.
\begin{lemma}\label{lm:area_asymptotics} Let $q_{n}$ be a sequence of holomorphic quartic differentials such that 
\begin{enumerate}
    \item the area $\|q_{n}\|$ diverges to $+\infty$;
    \item the renormalized metrics $|q_{n}|^{\frac{1}{2}}/\|q\|$ converge to a non purely laminar mixed structure $\eta$.
\end{enumerate}
Then,
\[
    \lim_{n\to +\infty} \frac{\Area(S,g_{n})}{\Area(S,4|q_{n}|^{\frac{1}{2}})}=1 \ .
\]
\end{lemma}
\begin{proof} By Proposition \ref{prop:area_coarse}, we already know that 
\[
    \frac{\Area(S,g_{n})}{\Area(S,4|q_{n}|^{\frac{1}{2}})} \geq 1
\]
for every $n\in \mathbb{N}$. It is thus sufficient to show that for every $\delta>0$, there is $n_{0} \in \mathbb{N}$ such that for every $n\geq n_{0}$ we have
\begin{equation}\label{eq:claim_bound}
    \frac{\Area(S,g_{n})}{\Area(S,4|q_{n}|^{\frac{1}{2}})} \leq 1+3\delta \ .
\end{equation}
From the proof of Proposition \ref{prop:area_coarse}, the area of the induced metric can be expressed as sum of three terms
\[
    \Area(S,g_{n})=\int_{S}\frac{e^{-2\psi_{1,n}}\|q_{n}\|_{\sigma_{n}}^{2}}{2e^{\psi_{1,n}-\psi_{2,n}}} dA_{g_{n}}+ \int_{S}\frac{e^{2\psi_{2,n}}}{2e^{\psi_{1,n}-\psi_{2,n}}} dA_{g_{n}} + 2\pi|\chi(S)|
\]
and we already showed that the first integral can be bounded above by
\[
    \int_{S}\frac{e^{-2\psi_{1,n}}\|q_{n}\|_{\sigma_{n}}^{2}}{2e^{\psi_{1,n}-\psi_{2,n}}} dA_{g_{n}} \leq 2\|q_{n}\| \ .
\]
The estimate (\ref{eq:claim_bound}) follows if we prove that for $n$ sufficiently large 
\begin{equation}\label{eq:claim2}
    \frac{1}{8\|q_{n}\|}\int_{S} e^{3\psi_{2,n}-\psi_{1,n}} dA_{g_{n}} \leq \frac{1}{2}+3\delta \ . 
\end{equation}
The non purely laminar mixed structure $\eta$ furnishes a decomposition of $S$ into incompressible subsurfaces. We denote by $S'$ the union of all subsurfaces on which $|q_{n}|^{\frac{1}{2}}/\|q\|$ converges geometrically to a flat metric induced by a meromorphic quartic differential $q'$ and we let $S''=S\setminus S'$. Recall that $S'$ has unit area for the flat metric $|q'|^{\frac{1}{2}}$. Choose $\epsilon>0$ such that the $|q'|^{\frac{1}{2}}$ area of the union $B$ of the balls of radius $\epsilon$ centered at the singularities of $q'$ is at most $\delta/2$. We observe that
\[
    1=\frac{\Area(S,|q_{n}|^{\frac{1}{2}})}{\|q_{n}\|}=\frac{1}{\|q_{n}\|}\left(\Area(S'\setminus B,|q_{n}|^{\frac{1}{2}})+\Area(B,|q_{n}|^{\frac{1}{2}})+\Area(S'',|q_{n}|^{\frac{1}{2}})\right)
\]
and 
\[
    \Area(S'\setminus B,|q_{n}|^{\frac{1}{2}}/\|q_{n}\|) \to \Area(S'\setminus B,|q'|^{\frac{1}{2}}) \geq 1-\delta/2
\]
hence 
\[
    \Area(B,|q_{n}|^{\frac{1}{2}}/\|q_{n}\|)+\Area(S'',|q_{n}|^{\frac{1}{2}}/\|q_{n}\|) \leq \delta \ 
\]
for $n$ sufficiently large.
We now split the domain of integration 
\[
    \int_{S} e^{3\psi_{2,n}-\psi_{1,n}} dA_{g_{n}}=\int_{S''} e^{3\psi_{2,n}-\psi_{1,n}} dA_{g_{n}}+\int_{B} e^{3\psi_{2,n}-\psi_{1,n}} dA_{g_{n}}+\int_{S'\setminus B} e^{3\psi_{2,n}-\psi_{1,n}} dA_{g_{n}}
\]
and proceed to estimate each term.  
By Lemma \ref{lm:bound} and Proposition \ref{prop:area_coarse} we have
\begin{align*}
    \frac{1}{8\|q_{n}\|}\int_{S''} e^{3\psi_{2,n}-\psi_{1,n}} dA_{g_{n}}&\leq \frac{1}{6\|q_{n}\|}\Area(S'',g_{n}) \\
    & \leq \frac{1}{\|q_{n}\|}\Area(S'',|q_{n}|^{\frac{1}{2}})+\frac{\pi|\chi(S)|}{\|q_{n}\|} \leq \delta+\frac{\delta}{3}
\end{align*}
for $n$ sufficiently large as $\|q_{n}\| \to +\infty$. Similarly,
\begin{align*}
    \frac{1}{8\|q_{n}\|}\int_{B} e^{3\psi_{2,n}-\psi_{1,n}} dA_{g_{n}} & \leq \frac{1}{6\|q_{n}\|}\Area(B,g_{n})\\
    &\leq \frac{1}{\|q_{n}\|}\Area(B,|q_{n}|^{\frac{1}{2}})+\frac{\pi|\chi(S)|}{\|q_{n}\|} \leq \delta+ \frac{\delta}{3}
\end{align*}
for $n$ large enough. For the last integral, we rewrite it in terms of $u_{1,n}$ and $u_{2,n}$
\begin{align*}
    \frac{1}{8\|q_{n}\|}\int_{S'\setminus B} e^{3\psi_{2,n}-\psi_{1,n}}dA_{g_{n}} &= \frac{1}{2\|q_{n}\|}\int_{S'\setminus B} e^{3u_{2,n}-u_{1,n}}e^{u_{1,n}-u_{2,n}} dA_{q_{n}} \\
    &= \frac{1}{2\|q_{n}\|}\int_{S'\setminus B} e^{2u_{2,n}} dA_{q_{n}} \ .
\end{align*}
By Corollary \ref{cor:decay}, the sequence $u_{2,n}$ converges to 0 uniformly on $S'\setminus B$ and $dA_{q_{n}}/\|q_{n}\|$ converges to $dA_{q'}$, hence the integral can be bounded by
\begin{align*}
    \frac{1}{2\|q_{n}\|}\int_{S'\setminus B} e^{2u_{2,n}} dA_{q_{n}} &\leq \frac{1}{2}\left(1+\frac{2}{3}\delta\right)\Area(S'\setminus B, |q_{n}|^{\frac{1}{2}}/\|q_{n}\|) \\
    &\leq \frac{1}{2}\left(1+\frac{2}{3}\delta\right)\Area(S, |q_{n}|^{\frac{1}{2}}/\|q_{n}\|) =\frac{1}{2}+\frac{\delta}{3}
\end{align*}
for $n$ sufficiently large and the estimate in Equation \ref{eq:claim2} follows. 
\end{proof}

\subsection{Limits of induced metrics} By Proposition \ref{prop:neg_curv}, the induced metrics on maximal surfaces are negatively curved, thus we can embed $\Ind(S)$ into the space of geodesic currents. Given an induced metric $g$, we will denote by $L_{g}$ the associated geodesic current. We recall the two main feature of this current:
\begin{enumerate}[i)]
    \item for every curve $\gamma \in \pi_{1}(S)$, we have $\ell_{g}(\gamma)=i(L_{g},\delta_{\gamma})$, where $\ell_{g}(\gamma)$ denotes the length of the unique geodesic representative of $\gamma$ for $g$;
    \item $i(L_{g},L_{g})=\frac{\pi}{2}\Area(S,g)$ \ .
\end{enumerate}

\noindent We have now all the ingredients to prove our main theorem.
\begin{thm} \label{thm:limit_induced} Let $(\sigma_{n},q_{n})\in \mathcal{Q}^{4}(S)$ be a sequence leaving every compact set. Let $g_{n}$ be the corresponding sequence of induced metrics. Then there is a sequence of positive real numbers $t_{n}$ and a mixed structure $\eta$ so that $t_{n}L_{g_{n}}\to \eta$. 
\end{thm}
\begin{proof} We distinguish two cases, depending on whether the area $\|q_{n}\|$ is uniformly bounded or not. \\
\underline{\textit{First case:} $\sup\|q_{n}\|<\infty$.} By Proposition \ref{prop:area_coarse}, the self-intersection $i(L_{g_{n}},L_{g_{n}})$ is uniformly bounded. We notice that, because $(\sigma_{n},q_{n})$ leaves every compact set, the sequence of constant curvature metrics $\sigma_{n}$ must necessarily diverge. Otherwise, up to subsequences, we could assume that $\sigma_{n}\to \sigma \in \mathcal{T}(S)$ and we could write $q_{n}=\tau_{n}\tilde{q}_{n}$, where 
\[
    \tau_{n}=\|q_{n}\|_{\infty}:=\max_{S}\frac{|q_{n}|^{2}}{\sigma_{n}^{4}} \to +\infty
\]
and $\tilde{q}_{n}$ converges uniformly to a non-vanishing quartic differential $\tilde{q}_{\infty} \in \mathcal{Q}(S,\sigma_{\infty})$, as unit spheres of holomorphic differentials over the thick part of Teichm\"uller space are compact. But then we would have
\[
    \|q_{n}\|=\int_{S}|\tau_{n}|^{\frac{1}{2}}|\tilde{q}_{n}|^{\frac{1}{2}}=|\tau_{n}|^{\frac{1}{2}}\int_{S} \frac{|\tilde{q}_{n}|^{\frac{1}{2}}}{\sigma_{n}} dA_{\sigma_{n}} \to +\infty
\]
because 
\[
    \frac{|\tilde{q}_{n}|^{\frac{1}{2}}}{\sigma_{n}} dA_{\sigma_{n}} \to \frac{|\tilde{q}_{\infty}|^{\frac{1}{2}}}{\sigma_{\infty}} dA_{\sigma_{\infty}} \neq 0 \ ,
\]
which would contradict our assumption that $\sup\|q_{n}\|<\infty$. Therefore, the sequence $2\sigma_{n}$ of hyperbolic metrics in the conformal class of $g_{n}$ diverges. Then, by Lemma \ref{lm:domination_constant}, the length spectrum of $L_{g_{n}}$ is unbounded. Since $\Pp\mathcal{C}(S)$ is compact, there exists a sequence $t_{n}\to 0$ such that $t_{n}L_{g_{n}}\to L_{\infty}$. We easily deduce that $i(L_{\infty},L_{\infty})=0$, hence $L_{\infty}$ is a measured lamination, that we can interpret as a mixed structure with no flat parts. \\
\underline{\textit{Second case:} $\sup\|q_{n}\|=\infty$.} By Proposition \ref{prop:area_coarse}, the self-intersection of $L_{g_{n}}$ diverges as $4\|q_{n}\|$. Since every geodesic current has finite self-intersection, we need to rescale $L_{g_{n}}$ at least by $\frac{1}{2\sqrt{\|q_{n}\|}}$. Let us denote
\[
    \hat{L}_{g_{n}}=\frac{1}{2\sqrt{\|q_{n}\|}}L_{g_{n}} \ .
\]
If the length spectrum of $\hat{L}_{g_{n}}$ is still unbounded, then there is a sequence $t_{n}\to 0$ such that $t_{n}\hat{L}_{g_{n}} \to \hat{L}_{\infty}$, which has vanishing self-intersection, thus $L_{h_{n}}$ converges projectively to a measured lamination.

If the length spectrum of $\hat{L}_{g_{n}}$ is uniformly bounded, then by Proposition \ref{prop:sub_flat}, the length spectrum of the unit area flat metrics $|q_{n}|^{\frac{1}{2}}/\|q_{n}\|$ is uniformly bounded as well. Thus, from the proof of Theorem \ref{thm:limit_flat}, the geodesic currents $L_{q_{n}}$ converges in $\Pp\mathcal{C}(S)$ to a mixed structure $\mu$ that is not purely laminar. This furnishes an orthogonal decomposition (for the intersection form $i$)  of the surface $S$ into a collection of incompressible subsurfaces $\{S_{j}'\}_{j=1}^{m}$, obtained by cutting $S$ along the simple closed curves $\gamma_{i}$, for which $\mu$ is induced by a flat metric on each $S_{j}'$ and is a measure lamination on the complement. Moreover, we can assume that each simple closed curve $\gamma_{i}$ bounds at least one flat part, induced by a meromorphic quartic differential $\tilde{q}_{j}$. On each $S_{j}'$, by Corollary \ref{cor:conformal_factor}, the conformal factor of the induced metric $g_{n}$ with respect to the flat metric $4|q_{n}|$ converges to $1$ uniformly on compact sets outside the cone singularities of $|\tilde{q}_{j}|^{\frac{1}{2}}$. This implies that
\[
    \frac{g_{n}}{4\|q_{n}\|} \to |\tilde{q}_{j}|^{\frac{1}{2}}
\]
uniformly on compact sets outside the zeros and the poles of $\tilde{q}_{j}$. We deduce that on each $S_{j}'$, we have
\[
    \hat{L}_{\infty}=\lim_{n \to +\infty}\hat{L}_{g_{n}}=\lim_{n\to +\infty}\frac{1}{\|q_{n}\|}L_{q_{n}}=L_{\tilde{q}_{j}} \ ,
\]
because uniform convergence of metrics implies convergence in the length spectrum (\cite[Proposition 5.3]{Charles_dPSL}). In particular, we notice that
\[
    \lim_{n \to +\infty}i(\hat{L}_{g_{n}}, \delta_{\gamma_{i}})=0
\]
so that the same collection of curves $\gamma_{i}$ can be used for the orthogonal decomposition of $\hat{L}_{\infty}$. Therefore, we can write
\[
    \hat{L}_{\infty}=\sum_{j=1}^{m}L_{\tilde{q}_{j}}+\lambda \ ,
\]
where $\lambda$ is a geodesic current supported in the complement of $\bigcup_{j}S_{j}'$. We claim that $\lambda$ is a measure lamination: in fact by Lemma \ref{lm:area_asymptotics}
\begin{align*}
    \frac{\pi}{2}&=\frac{\pi}{2}\lim_{n\to +\infty}\frac{\Area(S,g_{n})}{\|q_{n}\|}
    =\lim_{n \to +\infty}i(\hat{L}_{g_{n}},\hat{L}_{g_{n}})\\
    &=i(\hat{L}_{\infty}, \hat{L}_{\infty})
    =\sum_{j=1}^{m} i(L_{\tilde{q}_{j}}, L_{\tilde{q}_{j}})+i(\lambda, \lambda)\\
    &=\lim_{n\to +\infty} \frac{1}{\|q_{n}\|}i(L_{q_{n}},L_{q_{n}})+i(\lambda, \lambda)=\frac{\pi}{2}+i(\lambda, \lambda) \ ,
\end{align*}
so $\lambda$ has vanishing self-intersection. 
\end{proof}

\begin{proof}[Proof of Theorem \ref{thmA}] By Theorem \ref{thm:limit_induced}, we know that $\partial \overline{\Ind(S)}\subseteq \Pp\Mix(S)$. Consider now the family of induced metrics $g_{t}$ associated to a ray $(\sigma, tq) \in \mathcal{Q}^{4}(S)$ for a fixed unit area quartic differential. By Lemma \ref{lm:bound} and the proof of Theorem \ref{thm:limit_induced}, we know that $L_{g_{t}}$ converges projectively to $L_{q}$ as $t\to +\infty$. Therefore, $\partial \overline{\Ind(S)}\supseteq \overline{\Flat_{4}(S)}=\Pp\Mix(S)$, which proves the Theorem.
\end{proof}

\section{Entropy degeneration}
In this section we study the behavior of the volume entropy of the maximal surfaces along rays of quartic differentials. We will show that the entropy is maximal at the Fuchsian locus and monotonically decreases to zero along a ray.\\

We start with some background material. Let $g$ be a Riemannian metric on $S$, possibly with cone singularities, and denote by $\tilde{g}$ its lift to the universal cover $\widetilde{S}$. Fix a base point $p_{0}$ on $\widetilde{S}$. The volume entropy of $g$ can be defined as
\[
    \mathrm{Ent}(g)=\limsup_{r\to +\infty} \frac{\log(\#\{ p \in \widetilde{S} \ | \ p \in (\pi_{1}(S)\cdot p_{0}) \cap B_{\tilde{g}}(p_{0}, r)\})}{r} \ , 
\]
where $\pi_{1}(S)\cdot p_{0}$ denotes the orbit of $p_{0}$ under the action of $\pi_{1}(S)$ by deck transformations. It follows immediately from the definition that the volume entropy enjoys a scaling property: for any $t>0$,
\[
    \mathrm{Ent}(tg)=t^{-1}\mathrm{Ent}(g) \ ,
\]
and a monotonicity property: if $h$ is another Riemannian metric such that $h\leq g$, then 
\[
        \mathrm{Ent}(g)\leq \mathrm{Ent}(h) \ . 
\]
These facts are already sufficient to show the asymptotic behavior of the volume entropy of the maximal surfaces along a ray:
\begin{prop} Let $g_{t}$ be the induced metrics on the maximal surfaces associated to the parameters $(\sigma, tq) \in \mathcal{Q}^{4}(S)$ for some $q\neq 0$ and $t\in \R$. Then 
\[
    \lim_{t\to +\infty} \mathrm{Ent}(g_{t})=0 \ . 
\]
\end{prop}
\begin{proof} By Proposition \ref{prop:sub_ineq}, we know that $g_{t}\geq 4t^{\frac{1}{2}}|q|^{\frac{1}{2}}$, hence
\[
    0\leq \mathrm{Ent}(g_{t}) \leq \mathrm{Ent}(4t^{\frac{1}{2}}|q|^{\frac{1}{2}})
    =t^{-\frac{1}{2}}\mathrm{Ent}(4|q|^{\frac{1}{2}}) \xrightarrow{t\to +\infty} 0 \ .
\]
\end{proof}

In order to improve this result to Theorem \ref{thmC}, we need the following useful formula for the variation of the volume entropy along a path of smooth Riemannian metrics: 

\begin{thm}[\cite{variation_entropy}]\label{thm:variation_entropy} Let $g_{t}$ be a smooth path of negatively curved Riemannian metrics on a closed manifold $S$. Then 
\[
    \frac{d}{dt}\mathrm{Ent}(g_{t})_{|_{t=t_{0}}}=-\frac{\mathrm{Ent}(g_{t_{0}})}{2}\int_{T^{1}S}\frac{d}{dt}g_{t}(v,v)_{|_{t=t_{0}}} d\mu_{t_{0}} \ ,
\]
where $\mu_{t_{0}}$ denotes the Bowen-Margulis measure on the unit tangent bundle $T^{1}S$ of $(S,g_{t_{0}})$. 
\end{thm}

\noindent Because the induced metrics $g_{t}$ can be written as $g_{t}=4e^{\psi_{1}^{t}-\psi_{2}^{t}}\sigma$, with $\psi_{1}^{t}$ and $\psi_{2}^{t}$ being the solutions to the system
\begin{equation}\label{eq:PDE_ray}
\begin{cases}
\Delta_{\sigma} \psi_{1}^{t} =e^{\psi_{1}^{t} - \psi_{2}^{t}}-e^{-2\psi_{1}^{t}} \|tq\|^{2}_{\sigma} +\frac{3}{4}\kappa(\sigma) \\ 
\Delta_{\sigma} \psi_{2} =e^{2\psi_{2}^{t}}- e^{\psi_{1}^{t} - \psi_{2}^{t}} + \frac{1}{4}\kappa(\sigma)
\end{cases} \ ,
\end{equation}
Theorem \ref{thmC} follows if we show that $\psi_{1}^{t}-\psi_{2}^{t}$ depends smoothly on $t$ and its derivative is always positive. This is the content of the next two lemmas.

\begin{lemma}\label{lm:smooth_dependence} For any $(\sigma, q)\in \mathcal{Q}^{4}(S)$, the solutions $(\psi_{1}^{t}, \psi_{2}^{t})$ to the system (\ref{eq:PDE_ray}) depend smoothly on $t\in \R$. 
\end{lemma}
\begin{proof} Consider the map $F:C^{2,\alpha}(S)\times C^{2,\alpha}(S)\times \R \rightarrow C^{0,\alpha}(S)\times C^{0,\alpha}(S)$ defined by
\[
    F(u_{1},u_{2},t)=\begin{pmatrix}
    \Delta_{\sigma}u_{1}-e^{u_{1}-u_{2}}+e^{-2u_{1}}\|tq\|^{2}_{\sigma}-\frac{3}{4}\kappa(\sigma) \\
    \Delta_{\sigma}u_{2}-e^{2u_{2}}+e^{u_{1}-u_{2}}-\frac{1}{4}\kappa(\sigma)
     \end{pmatrix} \ .
\]
Note that the solution $(\psi_{1}^{t},\psi_{2}^{t})$ to Equation (\ref{eq:PDE_ray}) satisfies $F(\psi_{1}^{t},\psi_{2}^{t},t)=0$. At any $t_{0}\in \R$, the linearization
\begin{align*}
    &dF_{(\psi_{1},\psi_{2})}(\psi_{1}^{t_{0}},\psi_{2}^{t_{0}},t_{0}): C^{2,\alpha}(S)\times C^{2,\alpha}(S)\rightarrow C^{0,\alpha}(S)\times C^{0,\alpha}(S) \\
    \begin{pmatrix} \dot{\psi}_{1} \\ \dot{\psi}_{2}
    \end{pmatrix}& \mapsto  \left[\Delta_{\sigma}+
    \begin{pmatrix}
    -e^{\psi_{1}^{t_{0}}-\psi_{2}^{t_{0}}}-2e^{\psi_{1}^{t_{0}}}\|t_{0}q\|^{2} & e^{\psi_{1}^{t_{0}}-\psi_{2}^{t_{0}}} \\
    e^{\psi_{1}^{t_{0}}-\psi_{2}^{t_{0}}} & 
    -2e^{2\psi_{2}^{t_{0}}}-e^{\psi_{1}^{t_{0}}-\psi_{2}^{t_{0}}}
    \end{pmatrix}\right]\begin{pmatrix} \dot{\psi}_{1} \\ \dot{\psi}_{2}
    \end{pmatrix}
\end{align*}    
is a negative definite operator (\cite{lopez-gomez1994}), hence invertible. By the Implicit Function Theorem for Banach spaces, the solution to Equation (\ref{eq:PDE_ray}) depends smoothly on $t$ in a neighborhood of $t_{0}$. 
\end{proof}

\begin{lemma}\label{lm:monotonicity} For every $t_{0}>0$, we have
\[
    \frac{d}{dt}(\psi_{1}^{t}-\psi_{2}^{t})_{|_{t=t_{0}}}:=\dot{\psi}_{1}^{t_{0}}-\dot{\psi}_{2}^{t_{0}}>0
\]
at every point on $S$. 
\end{lemma}
\begin{proof} It is shown in \cite[Proposition 4.1]{QL_cyclic} that, if $0<t_{1}<t_{2}$, then $g_{t_{1}}<g_{t_{2}}$ pointwise. At any point $p\in S$, the function $(\psi_{1}^{t}-\psi_{2}^{t})(p)$ is then increasing for $t\in \R^{+}$ and smooth by Lemma \ref{lm:smooth_dependence}, thus it has positive first derivative.
\end{proof}

\begin{proof}[Proof of Theorem \ref{thmC}] For every unit tangent vector $v$, we have
\[
    \frac{d}{dt}g_{t}(v,v)_{|_{t=t_{0}}}=4e^{\dot{\psi}_{1}^{t_{0}}-\dot{\psi}_{2}^{t_{0}}}g_{t_{0}}(v,v)>0 
\]
by Lemma \ref{lm:monotonicity}. Because the induced metrics $g_{t}$ on the maximal surfaces are negatively curved by Proposition \ref{prop:neg_curv}, we can apply Theorem \ref{thm:variation_entropy} and deduce that the volume entropy of the maximal surfaces strictly decreases along a ray, since the Bowen-Margulis measure is positive on all non-empty open sets.
\end{proof}

\bibliographystyle{alpha}
\bibliography{bs-bibliography}

\newcommand{\etalchar}[1]{$^{#1}$}
\begin{thebibliography}{LGMM94}

\bibitem[Ama76]{Amann}
Herbert Amann.
\newblock Fixed point equations and nonlinear eigenvalue problems in ordered
  {B}anach spaces.
\newblock {\em SIAM Rev.}, 18(4):620--709, 1976.

\bibitem[Bar09]{Baraglia_thesis}
David Baraglia.
\newblock G2 geometry and integrable systems.
\newblock {\em arXiv:1002.1767}, 2009.

\bibitem[Bar15]{Baraglia_cyclic}
David Baraglia.
\newblock Cyclic {H}iggs bundles and the affine {T}oda equations.
\newblock {\em Geom. Dedicata}, 174:25--42, 2015.

\bibitem[BCG{\etalchar{+}}19]{strata}
Matt Bainbridge, Dawei Chen, Quentin Gendron, Samuel Grushevsky, and Martin
  M\"{o}ller.
\newblock Strata of {$k$}-differentials.
\newblock {\em Algebr. Geom.}, 6(2):196--233, 2019.

\bibitem[BIPP17]{BIPP}
Marc Burger, Alessandra Iozzi, Anne Parreau, and Maria~Beatrice Pozzetti.
\newblock A structure theorem for geodesic currents and length spectrum
  compactifications.
\newblock {\em arXiv:1710.07060}, 2017.

\bibitem[BIPP20]{burger2020real}
Marc Burger, Alessandra Iozzi, Anne Parreau, and Marie~Beatrice Pozzetti.
\newblock The real spectrum compactification of character varieties:
  characterizations and applications.
\newblock {\em arXiv preprint arXiv:2005.00120}, 2020.

\bibitem[BL18]{Rigidity_flat}
Anja Bankovic and Christopher~J. Leininger.
\newblock Marked-length-spectral rigidity for flat metrics.
\newblock {\em Trans. Amer. Math. Soc.}, 370(3):1867--1884, 2018.

\bibitem[Bon86]{bonahonbouts}
Francis Bonahon.
\newblock Bouts des vari\'et\'es hyperboliques de dimension {$3$}.
\newblock {\em Ann. of Math. (2)}, 124(1):71--158, 1986.

\bibitem[Bon88]{Bonahon_currents}
Francis Bonahon.
\newblock The geometry of {T}eichm\"{u}ller space via geodesic currents.
\newblock {\em Invent. Math.}, 92(1):139--162, 1988.

\bibitem[Col16]{Collier_thesis}
Brian Collier.
\newblock {\em Finite order automorphisms of {H}iggs {B}sundles: {T}heory and
  application}.
\newblock ProQuest LLC, Ann Arbor, MI, 2016.
\newblock Thesis (Ph.D.)--University of Illinois at Urbana-Champaign.

\bibitem[Cor88]{Corlette}
Kevin Corlette.
\newblock Flat {$G$}-bundles with canonical metrics.
\newblock {\em J. Differential Geom.}, 28(3):361--382, 1988.

\bibitem[CTT19]{BTT}
Brian Collier, Nicolas Tholozan, and J\'{e}r\'{e}my Toulisse.
\newblock The geometry of maximal representations of surface groups into {${\rm
  SO}_0(2,n)$}.
\newblock {\em Duke Math. J.}, 168(15):2873--2949, 2019.

\bibitem[DL20]{QL_cyclic}
Song Dai and Qiongling Li.
\newblock On cyclic {H}iggs bundles.
\newblock {\em Math. Ann.}, 376(3-4):1225--1260, 2020.

\bibitem[DLR10]{DLR_flat}
Moon Duchin, Christopher~J. Leininger, and Kasra Rafi.
\newblock Length spectra and degeneration of flat metrics.
\newblock {\em Invent. Math.}, 182(2):231--277, 2010.

\bibitem[Don87]{Donaldson_selfdual}
S.~K. Donaldson.
\newblock Twisted harmonic maps and the self-duality equations.
\newblock {\em Proc. London Math. Soc. (3)}, 55(1):127--131, 1987.

\bibitem[Fra12]{Frazier_thesis}
Jeffrey Frazier.
\newblock Length spectral rigidity of non-positively curved surfaces.
\newblock {\em arXiv:1207.6157}, 2012.

\bibitem[Gui08]{Guichard_Hitchin}
Olivier Guichard.
\newblock Composantes de {H}itchin et repr\'{e}sentations hyperconvexes de
  groupes de surface.
\newblock {\em J. Differential Geom.}, 80(3):391--431, 2008.

\bibitem[Gui18]{guichard2018introduction}
Olivier Guichard.
\newblock An introduction to the differential geometry of flat bundles and of
  higgs bundles.
\newblock {\em The Geometry, Topology and Physics of Moduli Spaces of Higgs
  Bundles, Lect. Notes Ser. Inst. Math. Sci. Natl. Univ. Singap}, 36:1--63,
  2018.

\bibitem[GW08]{GW_PSL4}
Olivier Guichard and Anna Wienhard.
\newblock Convex foliated projective structures and the {H}itchin component for
  {${\rm PSL}_4({\bf R})$}.
\newblock {\em Duke Math. J.}, 144(3):381--445, 2008.

\bibitem[Hit87]{Hitchin_selfdual}
N.~J. Hitchin.
\newblock The self-duality equations on a {R}iemann surface.
\newblock {\em Proc. London Math. Soc. (3)}, 55(1):59--126, 1987.

\bibitem[Hit92]{Hitchin_Teichmuller}
N.~J. Hitchin.
\newblock Lie groups and {T}eichm\"{u}ller space.
\newblock {\em Topology}, 31(3):449--473, 1992.

\bibitem[HP97]{Hers_Paulin}
Sa'ar Hersonsky and Fr\'{e}d\'{e}ric Paulin.
\newblock On the rigidity of discrete isometry groups of negatively curved
  spaces.
\newblock {\em Comment. Math. Helv.}, 72(3):349--388, 1997.

\bibitem[KKW91]{variation_entropy}
Anatole Katok, Gerhard Knieper, and Howard Weiss.
\newblock Formulas for the derivative and critical points of topological
  entropy for {A}nosov and geodesic flows.
\newblock {\em Comm. Math. Phys.}, 138(1):19--31, 1991.

\bibitem[Lab17]{Labourie_cyclic}
Fran\c{c}ois Labourie.
\newblock Cyclic surfaces and {H}itchin components in rank 2.
\newblock {\em Ann. of Math. (2)}, 185(1):1--58, 2017.

\bibitem[LGMM94]{lopez-gomez1994}
Julián López-Gómez and Marcela Molina-Meyer.
\newblock The maximum principle for cooperative weakly coupled elliptic systems
  and some applications.
\newblock {\em Differential Integral Equations}, 7(2):383--398, 1994.

\bibitem[Li19]{QL_survey}
Qiongling Li.
\newblock An introduction to higgs bundles via harmonic maps.
\newblock {\em SIGMA. Symmetry, Integrability and Geometry: Methods and
  Applications}, 15:035, 2019.

\bibitem[Lof10]{loftin2010survey}
J~Loftin.
\newblock Survey on affine spheres, handbook of geometric analysis, no. 2,
  2010.

\bibitem[LTW20]{LTW}
Fran\c{c}ois Labourie, J\'er\'emy Toulisse, and Michael Wolf.
\newblock Maximal surfaces in pseudo-hyperbolic spaces.
\newblock {\em arXiv:2006.12190}, 2020.

\bibitem[LW18]{LW_Fuchsianlocus}
Fran\c{c}ois Labourie and Richard Wentworth.
\newblock Variations along the {F}uchsian locus.
\newblock {\em Ann. Sci. \'{E}c. Norm. Sup\'{e}r. (4)}, 51(2):487--547, 2018.

\bibitem[Mar19]{martone2019sequences}
Giuseppe Martone.
\newblock Sequences of hitchin representations of tree-type.
\newblock {\em arXiv preprint arXiv:1910.13585}, 2019.

\bibitem[McM89]{McMullen_Poincareseries}
Curt McMullen.
\newblock Amenability, {P}oincar\'{e} series and quasiconformal maps.
\newblock {\em Invent. Math.}, 97(1):95--127, 1989.

\bibitem[MZ20]{Rafe_flat}
Rafe Mazzeo and Xuwen Zhu.
\newblock Conical metrics on {R}iemann surfaces {I}: {T}he compactified
  configuration space and regularity.
\newblock {\em Geom. Topol.}, 24(1):309--372, 2020.

\bibitem[Nie15]{Nie_rays}
Xin Nie.
\newblock Entropy degeneration of convex projective surfaces.
\newblock {\em Conform. Geom. Dyn.}, 19:318--322, 2015.

\bibitem[OT20]{OT}
Charles Ouyang and Andrea Tamburelli.
\newblock Limits of blaschke metrics.
\newblock {\em arXiv:1911.02119. To appear in Duke Math. J.}, 2020.

\bibitem[Ota90]{Otal}
Jean-Pierre Otal.
\newblock Le spectre marqu\'{e} des longueurs des surfaces \`a courbure
  n\'{e}gative.
\newblock {\em Ann. of Math. (2)}, 131(1):151--162, 1990.

\bibitem[Ouy19]{Charles_dPSL}
Charles Ouyang.
\newblock High energy harmonic maps and degeneration of minimal surfaces.
\newblock {\em arXiv:1910.06999}, 2019.

\bibitem[Sch93]{Schoenharmonic}
Richard~M. Schoen.
\newblock The role of harmonic mappings in rigidity and deformation problems.
\newblock In {\em Complex geometry ({O}saka, 1990)}, volume 143 of {\em Lecture
  Notes in Pure and Appl. Math.}, pages 179--200. Dekker, New York, 1993.

\bibitem[Sim92]{Simpson_Higgs}
Carlos~T. Simpson.
\newblock Higgs bundles and local systems.
\newblock {\em Inst. Hautes \'{E}tudes Sci. Publ. Math.}, (75):5--95, 1992.

\bibitem[Tam20]{Tambu_rays}
Andrea Tamburelli.
\newblock Degeneration of globally hyperbolic maximal anti-de {S}itter
  structures along rays.
\newblock {\em arXiv:1710.05827. To appear in Comm. Anal. Geom.}, 2020+.

\bibitem[TW20]{TW}
Andrea Tamburelli and Michael Wolf.
\newblock Planar minimal surfaces with polynomial growth in the
  {S}p(4,{R})-symmetric space.
\newblock {\em arXiv:2002.07295}, 2020.

\bibitem[Wen16]{wentworth2016higgs}
Richard Wentworth.
\newblock Higgs bundles and local systems on riemann surfaces.
\newblock In {\em Geometry and quantization of moduli spaces}, pages 165--219.
  Springer, 2016.

\bibitem[Wie18]{Wienhard_ICM}
Anna Wienhard.
\newblock An invitation to higher \text{T}eichm\"uller theory.
\newblock {\em Proc. Int. Cong. of Math. Rio de Janeiro 1}, pages 1007--1034,
  2018.

\bibitem[Wol89]{Wolf_harmonic}
Michael Wolf.
\newblock The {T}eichm\"uller theory of harmonic maps.
\newblock {\em J. Differential Geom.}, 29(2):449--479, 1989.

\end{thebibliography}

\bigskip
\noindent \footnotesize \textsc{Department of Mathematics and Statistics, University of Massachusetts, Amherst}\\
\emph{E-mail address:}  \verb|ouyang@math.umass.edu|

\bigskip
\noindent \footnotesize \textsc{Department of Mathematics, Rice University}\\
\emph{E-mail address:} \verb|andrea_tamburelli@libero.it|

\end{document}